                    \def\version{23 August, 2021}                       %

 \documentclass[reqno,11pt]{amsart}
 \usepackage{amsmath, amsthm, a4, latexsym, amssymb}
\usepackage[unicode]{hyperref}
\usepackage{srcltx}
\usepackage{xcolor}

\def\@rmrk#1#2{\refstepcounter
    {#1}\@ifnextchar[{\@yrmrk{#1}{#2}}{\@xrmrk{#1}{#2}}}

\makeatletter\@addtoreset{equation}{section}\makeatother

 \newfont{\bfit}{cmbxti10 scaled 1200}

 \newcommand{\e}{{\rm e} }

 \newcommand{\eps}{\varepsilon}

 \newcommand{\R}{\mathbb{R}}
 \newcommand{\N}{\mathbb{N}}
 \newcommand{\Z}{\mathbb{Z}}

 \newcommand{\E}{\mathbb{E}}
 \renewcommand{\P}{\mathbb{P}}
 \newcommand{\bbV}{\mathbb{V}}
 \def\1{{\mathchoice {1\mskip-4mu\mathrm l} 
{1\mskip-4mu\mathrm l}
{1\mskip-4.5mu\mathrm l} {1\mskip-5mu\mathrm l}}}

 \newcommand {\abs}[1] {\left\lvert {#1} \right\rvert}
\newcommand{\norm}[1]{\left\lVert {#1} \right\rVert}

\renewcommand{\subsection}{\secdef \subsct\sbsect}
\newcommand{\subsct}[2][default]{\refstepcounter{subsection}
\vspace{0.15cm}
{\flushleft\bf \arabic{section}.\arabic{subsection}~\bf #1  }
\nopagebreak\nopagebreak}
\newcommand{\sbsect}[1]{\vspace{0.1cm}\noindent
{\bf #1}\vspace{0.1cm}}

{\nopagebreak {\hfill\rule{2mm}{2mm}}\\ }

\newtheorem{theorem}{Theorem}[section]
\newtheorem{lemma}[theorem]{Lemma}
\newtheorem{cor}[theorem]{Corollary}
\newtheorem{prop}[theorem]{Proposition}

\newtheorem{definition}[theorem]{Definition} 

\newtheoremstyle{thm}{1.5ex}{1.5ex}{\itshape\rmfamily}{}
{\bfseries\rmfamily}{}{2ex}{}

\newtheoremstyle{rem}{1.3ex}{1.3ex}{\rmfamily}{}
{\itshape\rmfamily}{}{1.5ex}{}
\theoremstyle{rem}
\newtheorem{remark}{{\slshape\sffamily Remark}}[]

\refstepcounter{subsubsection}

\def\thebibliography#1{\section*{References}
  \list%
  {\arabic{enumi}.}
    {\settowidth\labelwidth{[#1]}\leftmargin\labelwidth
    \advance\leftmargin\labelsep
    \parsep0pt\itemsep0pt
    \usecounter{enumi}}
    \def\newblock{\hskip .11em plus .33em minus .07em}
    \sloppy                   
    \sfcode`\.=1000\relax}

\begin{document}
\title[Quenched and averaged large deviations for RWRE]
{\large Quenched and averaged large deviations for random walks in random environments: the impact of disorder}
\author[Rodrigo Bazaes, Chiranjib Mukherjee, Alejanro ram\'\i rez and Santiago Saglietti]{}
\maketitle
\thispagestyle{empty}
\vspace{-0.5cm}

\centerline{\sc By Rodrigo Bazaes\footnote{, Facultad de Matem\'aticas,Pontificia Universidad Cat\'olica de Chile,
Vicu\~na Mackenna 4860, Macul,
Santiago, Chile {\tt rebazaes@mat.uc.cl}},
Chiranjib Mukherjee\footnote{Fachbereich Mathematik und Informatik, Universit\"at M\"unster, Einsteinstrasse 62, M\"unster 48149 {\tt chiranjib.mukherjee@uni-muenster.de}},
Alejandro F. Ram\'\i rez\footnote{Facultad de Matem\'aticas, Pontificia Universidad Cat\'olica de Chile, Vicu\~na Mackenna 4860, Macul, Santiago, Chile {\tt aramirez@mat.uc.cl}}
and Santiago Saglietti\footnote{Facultad de Matem\'aticas, Pontificia Universidad Cat\'olica de Chile, Vicu\~na Mackenna 4860, Macul, Santiago, Chile
{\tt sasaglietti@mat.uc.cl}}}
\renewcommand{\thefootnote}{}
\footnote{\textit{AMS Subject
Classification:} 60K37, 60F10, 82C41}
\footnote{\textit{Keywords:} Random walks in random environment, large deviations, disorder, 
quenched and averaged rate functions, random walks in random scenery, disordered media.}

\vspace{-0.5cm}
\centerline{\textit{PUC Chile and Universit\"at M\"unster}}
\vspace{0.2cm}

\begin{center}
\version
\end{center}

\begin{quote}{\small {\bf Abstract: }
In 2003, Varadhan \cite{V03} developed a robust method for proving quenched and averaged large deviations 
for random walks in a uniformly elliptic and i.i.d. environment (RWRE) on $\Z^d$. One fundamental question which remained open
was to determine when the quenched and averaged large deviation rate functions agree, and when they do not.
In this article we show that for RWRE in uniformly elliptic and i.i.d. environment in $d\geq 4$, 
the two rate functions agree on any compact set contained in the interior of their domain which does not contain the origin, provided that the disorder of the environment is sufficiently low. 
Our result provides a new formulation which encompasses a set of sufficient 
conditions under which these rate functions agree without assuming that the RWRE is ballistic (see \cite{Y11}), satisfies a CLT or even a law of large numbers (\cite{Zer02,Ber08}). Also,
 the equality of rate functions is not restricted to neighborhoods around given points, as long as the disorder of the environment is kept low. One of the novelties of our approach is the introduction of an auxiliary random walk in a deterministic environment which is itself ballistic (regardless of the actual RWRE behavior) and whose large deviation properties approximate those of the original RWRE in a robust manner, even if the original RWRE is not ballistic itself.
}
\end{quote}


\section{Introduction and background.} 
Consider a random walk in an i.i.d. and uniformly elliptic random environment (RWRE) in $\Z^d$, $d\geq 1$. Multidimensional RWRE-s 
have remained a mathematically challenging topic -- in a general set up, some of its most fundamental questions like law of large numbers or CLTs have remained elusive till date. In this general set up and for any $d\geq 1$, Varadhan \cite{V03} showed that the rescaled location of the RWRE satisfies both quenched and averaged large deviation principles. A natural question which remained open was to determine when the quenched and averaged large deviation rate functions agree, and when they do not.  The main result of the current article is that, for $d\geq 4$ and {\it any} compact subset $\mathcal K$ of the open $\ell^1$-unit ball (not containing the origin), the quenched and averaged rate functions of {\it any} RWRE in a uniformly elliptic and i.i.d. environment agree on $\mathcal K$, if the {\it disorder} of the environment remains sufficiently small, see Theorem \ref{th:main}. Previously,  it was shown by Yilmaz \cite{Y11} that the two rate functions agree on some neighborhood of the non-zero limiting velocity in $d\geq 4$ whenever the RWRE is ballistic and satisfies Sznitman's condition (T). In contrast, our result does not require any ballisticity condition for the RWRE, nor do we need the RWRE to satisfy a CLT or even a law of large numbers (see Zerner \cite{Zer02} and Berger \cite{Ber08}); and the equality of rate functions is not restricted to neighborhoods around a given point, as long as the disorder is kept low. For example, the present set up covers the following RWRE models (where condition (T) is unavailable and the relations between the two rate functions have not been studied previously): (a) random walks in {\it balanced} random environments (see \cite{L82,GZ12,BD14}); (b) random walks in isotropic environments \cite{BK91,BZ07}; (c) environments which are perturbations of the simple random walk, invariant under reflections and balanced in one coordinate direction \cite{Ba16};  (d) and for RWRE models where the equivalent ballisticity conditions $(\mathrm P)_{\mathrm M} \Leftrightarrow (\mathrm T)_\gamma\Leftrightarrow(\mathrm T^\prime)\Leftrightarrow(\mathrm T)$ (see \cite{GR20} for the proof of this equivalence) fails to hold, in particular, including all the cases where neither the law of large numbers nor the existence of an asymptotic direction (see \cite{DR10}) has been proved.

Apart from the result itself, the present work introduces a novel point of view to study the problem of equality of the rate functions, namely that of the disorder of the environment. Indeed, our result suggests that, unless one is focused on particular regions of the domain (such as the corners in its boundary or neighborhoods around the velocity whenever the RWRE is ballistic), disorder should play an essential role in whether equality between the two rate functions holds, in the sense that equality should hold below and fail above a certain threshold disorder. This intuition has been confirmed when looking at the rate functions at the boundary of their domain for a certain wide family of environments in a separate work \cite{BMRS21}, see Remark \ref{rem:boundary}. 

The main technical contributions of the proof involves introducing a walk in a {\it deterministic environment}, 
comparing this to the original walk in the random environment and  controlling their Radon-Nikodym derivative in a robust manner uniformly over all environmental laws. The $d\geq 4$ assumption (in the absence of which the equality result does not hold) manifests in estimating the exponential tail on the size of the intersection of two random walks. Here the random walks in question are in the deterministic environments, and therefore the intersection estimate, contrary to previous works, does not depend on various ballisticity assumptions on the RWRE, see Section \ref{sec:proof:sketch} for an outline of the proof. Before turning to the precise statements, it is instructive to give some background on RWRE and underline some pertinent questions that motivated the current work.

RWRE-s provide a natural setting for studying  ``statistical mechanics in random media" and have  enjoyed a profound upsurge of interest in the last two decades within mathematicians and physicists. The one-dimensional model was first considered by Solomon \cite{So75}  and extended later by Sinai \cite{Si82} which provided a very efficient methodology 
which is by now fairly well-understood, and exhibits behaviors that are very different from that of the simple random walk. On the other hand, multi-dimensional RWRE turns out to be much more difficult to analyze than the one-dimensional model. 


\smallskip

\noindent The mathematical layout of RWRE can be described as a two-layer process. First, consider a sequence $\omega=(\omega(x))_{x \in \Z^d}$ of probability vectors on $\mathbb{V}:=\{x \in \Z^d : |x| = 1\}=\{\pm e_1,\dots,\pm e_d\}$ indexed by the sites of the lattice, i.e. $\omega(x)=(\omega(x,e))_{e \in \mathbb{V}}$ is a probability vector on $\mathbb{V}$ for each $x \in \Z^d$. Any such sequence $\omega$ will be called an \textit{environment} and the space $\Omega$ of all such sequences will be called the \textit{environment space}. Then, the first layer of our process consists of, for a fixed $\omega \in \Omega$, a random walk on the lattice whose jump probabilities are given by the environment $\omega$, i.e. for each $x \in \Z^d$ the law $P_{x,\omega}$ of this random walk $(X_n)_{n \geq 0}$ starting at $x$ is prescribed by
	$$
	P_{x,\omega}(X_0=x)=1 \hspace{1cm}\text{ and }\hspace{1cm}P_{x,\omega}(X_{n+1}=y+e\,|\,X_n=y) = \omega(y,e) \hspace{0.5cm}\forall\,y \in \Z^d\,,\,e \in \mathbb{V}.
	$$ We call $P_{x,\omega}$ the {\it quenched law} of the RWRE. The second layer of our process is then obtained when the environment $\omega$ is chosen at random according to some Borel probability measure $\P$ on $\Omega$ (when endowed with the usual product topology). We call any such $\P$ an \textit{environmental law}. Averaging $P_{x,\omega}$ over $\omega$ then produces a probability measure on $\Omega \times (\Z^d)^{\N_0}$ given by the formula
$$
P_x(A \times B) = \int_A P_{x,\omega}(B)\mathrm{d}\P \hspace{1cm}\forall\, A \in \mathcal{B}(\Omega)\,,\,B \in \mathcal{B}((\Z^d)^{\N_0}).
$$
We call the measure $P_x$ the \textit{averaged} or \textit{annealed} law of the RWRE (starting at $x$) and the sequence $X=(X_n)_{n \in \N_0}$ under $P_x$ a RWRE with environmental law $\P$. 
	
Given any RWRE, it is natural to ask whether classical limit theorems can hold for its quenched and annealed measures. The law of large numbers (LLN) for the quenched distribution, if valid, takes the form $\mathbb P\big(\omega\colon\,\, \lim_{n\to\infty} \frac {X_n}n =v\quad\mbox{a.e. w.r.t.}\,\, P_{x,\omega}\big)=1$ for some $v\in \mathbb R^d$. 
The latter display is equivalent to the validity of $P_x(\omega\colon\,\, \lim_{n\to\infty} {X_n}/n =v)=1$ which translates to the LLN for the annealed measure. 
We refer to the literature \cite{L82,PV81,K85,KV86,Zer02,Ber08,BZ08,RS09} where both LLN and central limit theorems (CLT) have been investigated quite successfully whenever the law $\mathbb P$ of the ambient environment enjoys some special properties like the existence of an invariant density for the {\it environment viewed from the particle} or that of {\it strong transience conditions.}

\smallskip

\noindent While RWRE exhibit the same behavior in the quenched and the annealed setting on the level of LLN, the resulting scenarios for the two cases could be very different for 
regimes concerning CLTs or large deviation principles (LDP). The latter statement concerns investigating the (formally written) asymptotic behavior 
\begin{equation}\label{QALDP}
\begin{aligned}
&\lim_{n\to\infty}\frac 1n \log P_{0,\omega}\bigg(\frac {X_n} n\approx x\bigg)\simeq - I_q(x), \quad\mbox{and} \quad 
&\,\,\lim_{n\to\infty}\frac 1n \log P_{0}\,\,\bigg(\frac {X_n} n\approx x\bigg)\simeq - I_a(x),
\end{aligned}
\end{equation}
where the former statement holds for $\mathbb P$-a.e. $\omega$, while $I_q$ and $I_a$ are the {\it quenched} and {\it annealed large deviation rate functions}, respectively. From Jensen's inequality and Fatou's lemma it follows that $I_a(\cdot)\leq I_q(\cdot)$. However, a deeper connection between the two rate functions is closely intertwined with  
 the profound interplay between the random walk and the underlying impurities of the environment, which leads to the following natural question: If at a large time $n$, the RWRE were to find itself at an atypical location, one could wonder if such  unlikely scenario resulted from a strange behavior of the particle in that environment or if the particle actually encountered an atypical environment. The answer to this question hinges upon a delicate statement regarding the equality of the two rate functions, a sufficient condition for which, as shown by our main result, is determined by the underlying disorder of RWRE models for which even the basic limit theorems are not required to hold. We turn to a precise statement of the main result of the article.

\section{Main result}
 In the sequel we shall work with environmental laws $\P$ satisfying the following assumption:

\noindent{\bf Assumption $\mathrm{A}$}: the environment is {\it i.i.d.} (i.e. the random vectors $(\omega(x))_{x \in \Z^d}$ are independent and identically distributed under $\mathbb P$) and \textit{uniformly elliptic} under $\P$, i.e., there is a constant $\kappa>0$ such that, for all $x\in\mathbb Z^d$ and $e \in \bbV$,
\begin{equation}
\label{eq:unifk}
\P(\omega(x,e)\geq\kappa)=1.
\end{equation}

Given any environmental law $\mathbb{P}$ satisfying Assumption $\mathrm{A}$, we now define its {\it disorder} as
\begin{align}
&\mathrm{dis}(\mathbb{P}):= \inf\big\{ \varepsilon > 0 :  \xi(x,e) \in [1-\varepsilon, 1+\varepsilon],\,\, \P\text{-a.s. for all }e \in \mathbb{V} \text{ and }x \in \Z^d \big\},\label{eq:defdis} \\
&\qquad\mbox{with }\,\, \xi(x,e):=\frac{\omega(x,e)}{\alpha(e)} \quad\mbox{and}\quad \alpha(e):= \E[\omega(x,e)] \qquad\forall\, e \in \mathbb V, \label{def-q} 
\end{align}
{where $\E$ denotes expectation w.r.t. $\P$ and the definition of $\alpha(e)$ does not depend on $x\in\Z^d$ by Assumption A.} Moreover, both
$\xi(x,e)$ and $\mathrm{dis}(\P)$ are well-defined since $\P$ satisfies Assumption $\mathrm{A}$ and 
$\mathrm{dis}(\mathbb{P})$ can be seen as the $L^\infty(\P)$-norm of the random vector $(\xi(x,e)-1)_{e \in \mathbb{V}}$ for any $x \in \Z^d$. 
In \cite{V03}, Varadhan proved that, under Assumption $\mathrm{A}$, both the quenched distribution $P_{0,\omega}\big(\frac {X_n}n\big)^{-1}$ 
and its averaged version $P_{0} \, \big(\frac {X_n}n\big)^{-1}$ satisfy a  large deviations principle, that is, that there exist two lower-semicontinuous functions $I_a,I_q:\mathbb R^d\to [0,\infty]$ such that for any $G\subset \mathbb R^d$ with interior $G^\circ$ and closure $\overline G$,
$$
-\inf_{x\in G^\circ}I_q(x)\le\liminf_{n\to\infty}
\frac{1}{n}\log P_{0,\omega}\left(
\frac{X_n}{n}\in G\right)\le
\limsup_{n\to\infty}
\frac{1}{n}\log P_{0,\omega}\left(
\frac{X_n}{n}\in G\right)\le-\inf_{x\in\overline G} I_q(x)
$$ for $\P$-almost every $\omega\in \Omega$ (and that the analogous statement obtained by replacing $P_{0,\omega}$ and $I_q$ by $P_0$ and $I_a$ also holds), see Remark \ref{rem:literature} for a brief overview of the literature on large deviations for RWRE. If $|x|_1$ denotes the $\ell^1$ norm of $x \in \R^d$, and we write $\mathbb D:=\{x\in \R^d\colon |x|_1\leq 1\}$ for the closed $\ell^1$-unit ball and $\mathrm{int}(\mathbb D):=\{x\in \R^d\colon |x|_1 < 1\}$ for its interior, it can be shown that the rate functions $I_q$ and $I_a$ are both convex and are finite if and only if $x\in \mathbb D$. Being also lower semicontinuous, this implies that both $I_q$ and $I_a$ are continuous functions on $\mathbb{D}$, see \cite[Theorem 10.2]{R97}. Furthermore, for any RWRE satisfying Assumption A, regardless of the disorder and in any $d\geq 2$, we always have $I_q(0)=I_a(0)$ and $\{I_q=0\}=\{I_a=0\}$ (see \cite[Theorem 8.1]{V03} and also Theorem 7.1 there for a formula for $I_a(0)$) and it is also well-known that one always has the inequality $I_a \leq I_q$. 
Here is our main result. 
\begin{theorem}\label{th:main}
	For any $d \geq 4$, $\kappa>0$ and compact set $\mathcal{K} \subseteq \mathrm{int}(\mathbb{D}) \setminus \{0\}$, there exists $\eps=\eps(d,\kappa,\mathcal{K}) >0$ such that, for any RWRE satisfying Assumption $\mathrm{A}$ with ellipticity constant $\kappa$, if  $\mathrm{dis}(\mathbb{P})<\eps$ then we have the equality 
	$$
	I_q(x)=I_a(x) \qquad \text{ for all }x \in \mathcal{K}.
	$$
\end{theorem}

Let us make some comments about the result.

\begin{remark}[The region of equality]\label{rem:theo:main}
As mentioned above, since $I_q(0)=I_a(0)$, for any $d\geq 4$ and $\kappa>0$ and for any $x\in \mathrm{int}(\mathbb D)$, the above result implies that 
there is $\eps>0$ such that $I_a(x)=I_q(x)$ for $\mathrm{dis}(\P)<\eps $, so we can think of the result above as saying that the region of equality $\{ x \in \text{int}(\mathbb{D}) : I_q(x)=I_a(x)\}$ covers the entirety of 
$\text{int}(\mathbb{D})$ in the limit as $\mathrm{dis}(\P) \rightarrow 0$, uniformly over all environmental laws $\P$ with a uniform ellipticity constant bounded from below by some $\kappa > 0$. However, we point out that, for a {\it fixed environmental law} $\P$ and regardless of its disorder, it follows easily from Jensen's inequality that 
$I_a$ and $I_q$ can never be equal everywhere on the boundary $\partial{\mathbb D}$ (unless $\P$ is degenerate), and by continuity of $I_a$ and $I_q$, the strict inequality $I_a(\cdot)< I_q(\cdot)$ then extends also to some regions in $\text{int}(\mathbb{D})$, see \cite[Proposition 4]{Y11}. Finally, for $d\in \{2,3\}$, such an identity between the two rate functions is not expected to be true for general RWRE, as shown in \cite{YZ10}: for $d=2,3$ there is a class of non-nestling random walks in uniformly elliptic and i.i.d. environments 
such that $I_a$ and $I_q$ are never identical on any open neighborhood of the velocity.\qed
 \end{remark}
  
 \begin{remark}[An auxiliary random walk]\label{rem:proof}
 One of the novelties of our approach is the introduction of an auxiliary random walk (in a deterministic environment) satisfying the following key properties: (i) (a particular version of) its logarithmic moment generating function is intimately related with those of the RWRE (see Section \ref{sec:proof:sketch} for further details) and (ii) this walk is ballistic and possesses a strong regeneration structure.  By means of this auxiliary walk, we are able to study the LDP properties of the original RWRE using techniques available for ballistic walks, even if our original RWRE is not ballistic itself.\qed
 \end{remark}

\begin{remark}[Literature remarks]\label{rem:literature}
Large deviations for RWRE for $d=1$ were handled by Greven and den Hollander \cite{GdH94} in the quenched setting and by Comets, Gantert and Zeitouni \cite{CGZ00} (see also \cite{GZ98}) in both quenched and annealed settings (including a variational formula relating the the two rate functions. For $d\geq 1$, using sub-additive arguments, Zerner \cite{Zer98} (see also Sznitman \cite{S94}) proved a quenched LDP for ``nestling environments"\footnote{A RWRE is called nestling if the origin lies in the interior of the convex hull of the support of the local drift  $\sum_{e \in \mathbb V} e \omega(0,e)$ around the origin.}, while Varadhan \cite{V03} dropped the latter assumption on the environment and proved both the quenched and annealed~LDP. Kosygina, Rezakhanlou and Varadhan \cite{KRV06} developed a novel method for obtaining quenched LDP for
elliptic diffusions with a random drift based on a convex variational approach (see also Kosygina-Varadhan \cite{KV08}) which was adapted by Rosenbluth \cite{R06} for elliptic RWRE in $d\geq 1$ and developed further by Yilmaz \cite{Y08} and by Rassoul-Agha and Sep\"al\"ainen \cite{RS11}. The latter approach was extended to non-elliptic models like random walks on percolation clusters including long-range correlations in  \cite{BMO16} (see also Kubota \cite{K12} and Mourrat \cite{M12} for sub-additive approaches to quenched large deviations, see \cite{SS04,BB07,MP07} for quenched CLT results). We also refer to \cite{GKS07} for relevant results on large deviations for random walks  in random scenery \cite{KS79}. \qed

\end{remark}

\begin{remark}[Previous results under condition-$(\mathrm T)$]\label{rem:condT}
To put our work into context, let us now comment on a strong ballisticity criterion known as {\it Condition-$(\mathrm T)$}, introduced by Sznitman~\cite{S01}, which is the main assumption for all previously known results on the equality of the rate functions (at least for standard RWRE in dimensions $d \geq 4$). Given a direction $\ell\in\mathbb S^{d-1}$, the RWRE is said to satisfy condition-$(\mathrm T)$ if for some $\gamma \in (0,1]$ (or, equivalently, if for any such $\gamma$) there exists a neighborhood $V$ of $\ell$ such that, for all $\ell^\prime \in V$,
\begin{equation}\label{ballisticity}
\lim_{L\to\infty} L^{-\gamma} \log P_0\big[\langle X_{T_{U_{\ell^\prime,L}}}, \ell^\prime\rangle<0\big]<0,
\end{equation}
\textup{where $T_{U_{\ell^\prime,L}}:=\inf\{n\geq 0\colon X_n\notin U_{\ell^\prime, L}\}$ is the exit time from $U_{\ell^\prime,L}:= \{x\in\mathbb Z^d: -L< \langle x, \ell^\prime\rangle < L\}$}, see also \cite{GR20}. 
Under Assumption $\mathrm{A}$, condition-$(\mathrm T)$ implies that: i) a law of large numbers $\lim_{n\to\infty}\frac{ X_n }{n}=v$ holds $P_0$-a.s. with a non-zero velocity $v$ and ii) there exist regeneration times (with finite moments) such that the RWRE segments embedded between  
these times are an i.i.d. sequence under $P_0$. This regeneration structure has proved very fruitful tool in the study of LDP for RWRE (see e.g. \cite{PZ09,Y10,Ber12}). However, there are prominent RWRE models which do not satisfy this condition; see below for some examples of these models that are included in our current set up. 

Under Assumption A and condition-$(\mathrm T)$, it was shown in \cite{Y11} that when $d\geq 4$, $I_a=I_q$ on some (possibly small) neighborhood of the non-zero velocity (which, as mentioned above, always exists under \eqref{ballisticity}). Note that this result does not require the disorder of the environment to be small, but in return only yields equality in a (possibly small) neighborhood of a very specific point in the domain. In contrast, Theorem \ref{th:main} does not require the walk to be ballistic nor are we restricted to neighborhoods around given points, as long as the disorder of the environment is maintained low. As mentioned earlier,  our result applies to the following models where Sznitman's condition-(T) is not available: random walks in {\it balanced} random environments (i.e. such that $\mathbb \P(\omega(x,e)=\omega(x,-e)\, \forall\, x,e)=1$), random walks in isotropic environments, environments which are perturbations of the simple random walk, invariant under reflections and balanced in one coordinate direction and for RWRE models where the equivalent ballisticity conditions $(\mathrm P)_{\mathrm M} \Leftrightarrow (\mathrm T)_\gamma\Leftrightarrow(\mathrm T^\prime)\Leftrightarrow(\mathrm T)$ fails to hold, in particular, including all the cases where neither the law of large numbers nor the existence of an asymptotic direction has been proved. Finally, we recall that, as shown in \cite{V03}, we always have the equality $I_a(0)=I_q(0)$ under Assumption $\mathrm{A}$, regardless of the disorder. However, except for some results in specific scenarios (for RWRE satisfying condition-(T) and which are {\it nestling}, see \cite[Theorem 5-(iv)]{Y11}), both our approach and that in \cite{Y11} seem unfit to study the equality of the rate functions in \textit{neighborhoods} of the origin.\qed
 
 \end{remark} 
 
 \begin{remark}[Relation between $I_q$ and $I_a$ on the boundary of $\mathbb D$]\label{rem:boundary}
In a separate work \cite{BMRS21}, we show the analogue of Theorem \ref{th:main} for compact sets on the boundary $\partial\mathbb D:=\{x\in \R^d: |x|_1=1\}$ (not intersecting any of the $(d-2)$-dimensional facets of $\partial \mathbb D$). As a consequence, we obtain that both $I_q$ and $I_a$ admit simple explicit formulas on the boundary $\partial\mathbb D$ for sufficiently small disorder. We refer to \cite{RSY17a,RSY17b} for an alternative variational representation of $I_q$.\qed
\end{remark}

\subsection{Outline of the proof}\label{sec:proof:sketch}

For conceptual clarity and convenience of the reader, we now present a brief outline of the proof of Theorem \ref{th:main}, highlighting the main technical contribution of our approach, and underlining the similarities and differences with earlier works under condition-$(\mathrm T)$.
	
The proof of Theorem \ref{th:main} consists of three parts, which we summarize below. We first notice that, in order to obtain Theorem \ref{th:main}, it will suffice to show that, for any $y \in \mathrm{int}(\mathbb{D}) \setminus \{0\}$ and $\kappa>0$, if $d \geq 4$ then there exist $\eps=\eps(y,d,\kappa),r=r(y,d,\kappa) >0$ such that $I_q=I_a$ on $B_r(y)$, the $\ell^1$-ball of radius $r$ centered at $y$, for any RWRE with $\mathrm{dis}(\P)<\varepsilon$ which satisfies Assumption $\mathrm{A}$ with ellipticity constant $\kappa$ (see Theorem \ref{th:main2} below). Thus, in the following we explain the steps towards showing this variant of Theorem \ref{th:main} for a henceforth fixed $y \in \mathrm{int}(\mathbb{D}) \setminus \{0\}$ and $\kappa>0$.

\noindent \underline{\bf Step 1}: The first building block of the proof, which is one of the main novelties of our approach, is the construction of an auxiliary random walk in a \textit{deterministic} environment verifying that:
\begin{enumerate}
	\item [Q1.] It is ballistic with velocity $y$ and, furthermore, possesses strong regeneration properties;
	\item [Q2.] If we denote its law when starting from $0$ by $Q_0$ and define its ``quenched'' limiting logarithmic moment generating functions as
\[
\overline{\Lambda}_{q}(\theta):=\lim_{n\to \infty}\frac{1}{n}\log E^Q_0\Big(\e^{\langle \theta, X_{n}\rangle} \prod_{j=1}^{n}\xi(X_{j-1},\Delta_j(X))\Big) \quad \theta\in \R^d,
\] 
and its ``averaged" counterpart $\overline{\Lambda}_{a}(\theta)$ being defined analogously with $\prod_{j=1}^{n}\xi(X_{j-1},\Delta_j(X))$ replaced by its averages $\E\prod_{j=1}^{n}\xi(X_{j-1},\Delta_j(X))$ (here $E^Q_0$ denotes expectation w.r.t. $Q_0$, $\omega$ is the random environment from our original RWRE and $\xi$ is given by \eqref{def-q}), then essentially (see Section \ref{sec:dev} for details)
\begin{equation}
\label{eq:igualdad}
I_q-I_a = \tilde{I}_q -\tilde{I}_a,
\end{equation} where $\tilde{I}_q$ and $\tilde{I}_a$ are the Fenchel-Legendre transforms of $\overline{\Lambda}_q$ and $\overline{\Lambda}_a$ respectively, i.e.
\[
\tilde{I}_q(x)=\sup_{\theta \in \R^d}[\langle\theta,x\rangle - \overline{\Lambda}_q(\theta)]\hspace{1cm}\text{ and }\hspace{1cm}\tilde{I}_a(x)=\sup_{\theta \in \R^d}[\langle\theta,x\rangle - \overline{\Lambda}_a(\theta)].
\]
\end{enumerate}
Thus, we see from \eqref{eq:igualdad} that, in order to establish that $I_q=I_a$, it will suffice to show that $\tilde{I}_q=\tilde{I}_a$. Noting that $\tilde{I}_q$ and $\tilde{I}_a$ are essentially ``quenched'' and ``averaged'' versions of a random perturbation determined by $\xi$ of the rate function for this auxiliary walk.  
 Now, in light of (Q1) above, one could try to adapt the method from \cite{Y11} originally devised for RWRE with strong regeneration properties to our auxiliary walk in order to show that $\tilde{I}_q=\tilde{I}_a$. However, note that these two settings are not the same -- indeed, we have a {\it deterministic environment} as opposed to a random one as in \cite{Y11}, and we work with random perturbations of logarithmic MGFs instead of actual MGFs. Thus, one of the main challenges of our work is to 
control~these random perturbations well enough, and we must do so  \textit{uniformly} over all environmental laws with a uniform ellipticity constant bounded from below by $\kappa$
in order to leverage properties of the auxiliary walk. We outline all the necessary steps next.

\noindent \underline{\bf Step 2}: As stated above, we must prove that there exist $\varepsilon, r >0$, depending only on $y,d$ and $\kappa$, such that $\tilde{I}_a(x)=\tilde{I}_q(x)$ for all $x \in B_r(y)$ and any RWRE with environmental law $\P$ with $\mathrm{dis}(\P)<\varepsilon$. As a first step towards this, we show that $\overline{\Lambda}_q(\theta)=\overline{\Lambda}_a(\theta)$ for all $\theta$ with $|\theta|<r_1$ and if $\mathrm{dis}(\P) < \varepsilon_1$, for $r_1, \varepsilon_1 >0$ depending only on $y,d$ and $\kappa$. The main step for showing this is establishing the $L^2(\P)$-boundedness of a particular sequence $(\Phi_n)_{n \in \N}$, which is closely related (if not equal) to a martingale. In our case, the sequence of interest is 
\begin{equation}\label{eq:defphintro}
	\Phi_n(\theta,\omega)= \overline E^Q_0\bigg(\e^{\langle \theta, X_{L_n}\rangle- \overline \Lambda_a(\theta) L_n}\, \, \prod_{j=1}^{L_n} \xi(X_{j-1}, X_j- X_{j-1})\,;\, L_n \text{ is a regeneration time}\bigg),
\end{equation} where $L_n$ denotes the hitting time of the hyperplane $\{ x : \langle x,\ell\rangle = n\}$ and $\overline{E}^Q_0$ denotes expectation w.r.t. $Q_0$ conditional on the event that $\inf\{ n : \langle X_n - X_0,\ell \rangle < 0\}=\infty$ for $\ell \in \mathbb{V}$ some particular direction satisfying that $\langle y,\ell \rangle >0$ (and in terms of which the regeneration structure of the auxiliary random walk is defined, see Section \ref{sec:auxiliary} for details). To see that $(\Phi_n)_{n \in \N}$ is bounded in $L^2(\P)$, one writes $\Phi_n^2$ as an expectation of a certain function of two independent random walks with law $Q_0$ and argues that, in order to bound its second moment whenever the disorder is sufficiently small, it is enough to suitably estimate the number of times the trajectories of these two independent walks intersect and then use this to control the perturbation term $\prod_{j=1}^{L_n} \xi(X_{j-1},X_j-X_{j-1})$ from~\eqref{eq:defphintro}. The desired control reduces to a suitable bound for the probability of non-intersection of two random walks in the same deterministic environment. For this purpose, and in contrast to previous approaches where an estimate by Berger-Zeitouni \cite{BZ08} has been used (to control the probability of non-intersection of two walks in the same {\it random environment}), we invoke the bounds developed by Bolthausen and Sznitman \cite{BS02} which are better suited to our setting.

\noindent \underline{\bf Step 3}: The last task in the proof is to show that the equality of $\overline{\Lambda}_q$ and $\overline{\Lambda}_a$ in a neighborhood of the origin translates into the equality of $\tilde{I}_q$ and $\tilde{I}_q$ in a neighborhood of $y$. To do this, we note~that, by standard arguments, we have that for any $x \in \R^d$,
\[
\tilde{I}_q(x)=\langle \theta_{q,x},x\rangle -\overline{\Lambda}_q(\theta_{x,q}) \hspace{1cm}\text{ and }\hspace{1cm}\tilde{I}_a(x)=\langle \theta_{a,x},x\rangle - \overline{\Lambda}_a(\theta_{x,a})
\] for any $\theta_{x,q},\theta_{x,a} \in \R^d$ such that $\nabla \overline{\Lambda}_q(\theta_{x,q}) = x = \nabla \overline{\Lambda}_a(\theta_{x,a})$. In particular, if we can take $\theta_{x,q}=\theta_{x,a}$ then this readily implies that $\tilde{I}_q(x)=\tilde{I}_a(x)$. Thus, since $\nabla \overline{\Lambda}_q(\theta)=\nabla \overline{\Lambda}_a(\theta)$ for $|\theta|< r_1$ if $\mathrm{dis}(\P)<\varepsilon_1$ by Step 2, if we show that there exist $0 < \varepsilon(y,d,\kappa) < \varepsilon_1$ and $r(y,d,\kappa)>0$ such~that 
\begin{equation} \label{eq:unifinc}
B_r(y) \subseteq \{\nabla \overline{\Lambda}_a(\theta) : |\theta|<r_1\}
\end{equation} whenever $\mathrm{dis}(\P)<\varepsilon$ then for each $x \in B_r(y)$ we would have $\nabla \overline{\Lambda}_q(\theta_{x}) = x = \nabla \overline{\Lambda}_a(\theta_{x})$ for some $\theta_x$ and hence that $\tilde{I}_q(x)=\tilde{I}_a(x)$ immediately follows. A key point here is that we must show that $r$ in \eqref{eq:unifinc} can be taken to be \textit{independent} of the law $\P$, as long as its disorder is sufficiently~low and its uniformly ellipticity constant is bounded from below by $\kappa$. We achieve this by using a \textit{uniform} inverse function theorem for families of differentiable functions (Theorem \ref{prop3} below), which requires us to obtain uniform estimates (over $\P$) on the modulus of continuity at $\theta=0$ of the Hessian $H_a$ of $\overline{\Lambda}_a$ as well as a uniform upper bound on the norm of its inverse $H_a^{-1}$. This concludes the outline of the proof of Theorem \ref{th:main}. 

\subsection{Organization of the article:} The rest of the paper is organized as follows. The construction of the auxiliary random walk as well as the study of its properties is carried out in Section \ref{sec:auxiliary}. Also in Section \ref{sec:auxiliary} the reader will find proof of the equality $\overline{\Lambda}_q=\overline{\Lambda}_a$ in a neighborhood of the origin, assuming the $L^2(\P)$-boundedness of the sequence $(\Phi_n)_{n \in \N}$ in \eqref{eq:defphintro}, the proof of which is deferred to Section \ref{sec:l12}. Finally, Step 3 in the above discussion is carried out in Section \ref{sec:Hessian}.

\section{An auxiliary random walk and equality of its limiting log-MGFs}\label{sec:auxiliary}

As stated in Section \ref{sec:proof:sketch}, Theorem \ref{th:main} is a direct consequence of the following more specific result.

\begin{theorem}\label{th:main2}
	For any $y \in \mathrm{int}(\mathbb{D}) \setminus \{0\}$, $d \geq 4$ and $\kappa>0$, there exist $\eps=\eps(y,d,\kappa),r=r(y,d,\kappa) >0$ such that, for any RWRE satisfying Assumption $\mathrm{A}$ with ellipticity constant $\kappa$, if  $\mathrm{dis}(\mathbb{P})<\eps$ then then we have the equality 
	\[
	I_q(x)=I_a(x) \qquad \text{ for all }x \in B_r(y):=\{ z \in \R^d : |z-y| < r\}.
	\]
\end{theorem} Therefore, here and in the coming sections we shall focus only on proving Theorem \ref{th:main2}.
The goal in this particular section is to begin the proof by showing equality between the averaged and quenched limiting logarithmic moment generating functions (log-MGFs, for short) for small enough disorder. A key building block to this end will be the construction of an auxiliary random walk and a detailed investigation of its properties. 

Before we begin we introduce some further notation to be used throughout the sequel. 
Given $\kappa > 0$, we define 
\[
\mathcal{M}_1^{(\kappa)}(\bbV):= \{ p \in \mathcal{M}_1(\bbV) : p(e)\geq \kappa \text{ for all }e \in \bbV\}
\] with $\mathcal{M}_1(\mathbb{V})$ the space of all probability vectors on $\mathbb{V}$, together with the class of environmental laws 
\[
\mathcal{P}_\kappa:=\{ \P \in \mathcal{M}_1(\Omega) : \P \text{ satisfies Assumption $\mathrm{A}$ with ellipticity constant $\kappa$}\},
\] where $\mathcal{M}_1(\Omega)$ is the space of all environmental laws. Finally, for $\eps > 0$, we define
\[
\mathcal{P}_\kappa(\eps):= \{ \P \in \mathcal{P}_\kappa : \mathrm{dis}(\P) < \eps \}.
\] We are now ready to present this auxiliary random walk and study its properties.

\subsection{Introducing the $Q$-random walk and its limiting log-MGFs}

Let us fix $y\in \mathrm{int}(\mathbb{D})\setminus \{0\}$ and $\mathbb{P} \in \mathcal{P}_\kappa$. Notice that, if we define the function $f:[0,\infty) \to [0,\infty)$ as
\begin{equation}\label{def:f}
f(C):=\sum_{i=1}^d \sqrt{|\langle y,e_i\rangle|^2 +4C\alpha(e_i)\alpha(-e_i)}
\end{equation} then, since $f$ is strictly increasing and continuous, with $f(0)=|y| < 1$ and $\lim_{C \to \infty} f(C)=\infty$, there exists a unique $C_{y,\alpha} \in (0,\infty)$ such that $f(C_{y,\alpha})=1$. With this, we may define for each $e \in \mathbb{V}$ the probability weight
\begin{equation}
u(e):=\frac{\langle y,e\rangle}{2}+\frac{1}{2}\sqrt{|\langle  y,e\rangle|^{2}+4C_{y,\alpha}\alpha(e)\alpha(-e)}\label{eq:w1}
\end{equation} Observe that $u(e) \geq 0$ and $\sum_{e \in \mathbb{V}} u(e)=1$, so that $u:=(u(e))_{e \in \mathbb{V}}$ truly is a probability vector. Central to the proof of Theorem \ref{th:main2} will be the following auxiliary random walk (in a deterministic environment) on $\Z^d$, whose law we denote by $Q$, which is given by the transition probabilities 
\[
Q(X_{n+1}= x + e \,|\, X_n = x) = u(e)  
\] for each $e \in \mathbb{V}$ and $x \in \Z^d$, with $u(e)$ as in \eqref{eq:w1}. We call this auxiliary walk the $Q$-\textit{random walk}. We will write $Q_x$ to denote the law of this walk starting from a fixed $x \in \Z^d$ and $E^Q_x$ to denote expectations with respect to $Q_x$.
Notice that $Q_x$ depends exclusively on $x$, $y$ and $\P$, but it depends on $\P$ only through the average weights $\alpha$. In general, we will omit the dependence on $y$ and $\alpha$ from the notation, but occasionally we will write $Q(y,\alpha)$ instead of $Q$ if we wish to make it explicit. Furthermore, the weights $u$ have been particularly chosen so that this $Q$-random walk satisfies the properties in Lemma \ref{lemma:prop} below.

\begin{lemma}\label{lemma:prop} With this choice of probability weights $u=(u(e))_{e \in \mathbb{V}}$, the following properties hold:
	\begin{enumerate}
		\item [P1.] Given $\kappa >0$ there exists $c_\kappa >0$ such that $u(e) \geq c_\kappa$ for all $e \in \bbV$ and $\P \in \mathcal{P}_\kappa$. 
		\item [P2.] $E^Q_x(X_{n+1}-X_n)=y$ for all $n \in \N$ and $x \in Z^d$.
		\item [P3.] For any $x\in \Z^{d}$ and all environments $\omega$, we have
		\begin{equation}\label{eq:rep1a}
		E^Q_0\Big(\e^{\langle \theta, X_{n}\rangle}\prod_{j=1}^{n}\xi(X_{j-1},\Delta_j (X))\Big)=(C_{y,\alpha})^{\tfrac{n}{2}} E_{0,\omega}\big(\e^{\langle \theta+\theta_{y,\alpha},X_{n}\rangle}\big)
		\end{equation} and 
		\begin{equation}\label{eq:rep2a}
		E^Q_0\Big(\e^{\langle \theta, X_{n}\rangle}\E\prod_{j=1}^{n}\xi(X_{j-1},\Delta_j (X))\Big)=(C_{y,\alpha})^{\tfrac{n}{2}} E_{0}\big(\e^{\langle \theta+\theta_{y,\alpha},X_{n}\rangle}\big),
		\end{equation}
		where $C_{y,\alpha}$ is as in \eqref{eq:w1}, the vector $\theta_{y,\alpha} \in \R^d$ is given by the formulas 
		\begin{equation}\label{eq:deftheta}
		\langle \theta_{y,\alpha},e_i\rangle :=\log\left(\frac{u(e_i)}{\alpha(e_i)\sqrt{C_{y,\alpha}}}\right) \qquad i=1,\dots,d
		\end{equation} and we use the notation $\Delta_j(X):=X_j-X_{j-1}$ for $j=1,\dots,n$. 
	\end{enumerate}
\end{lemma}

\begin{proof} Since the mapping $\alpha \mapsto C_{y,\alpha}$ is continuous on $\mathcal{M}_1^{(\kappa)}(\bbV)$ (by the implicit function theorem, for example), we see that $\alpha \mapsto u(e)$ is also continuous for each $e \in \bbV$. In particular, since $\mathcal{M}_1^{(\kappa)}(\bbV)$ is compact, we see that $\inf_{\P \in \mathcal{P}_\kappa} u(e) = \inf_{\alpha \in \mathcal{M}^{(\kappa)}_1(\bbV)} u(e) >0$ for each $e \in \bbV$, which readily implies (P1). On the other hand, (P2) is immediate from the definition of the weights $u$ in \eqref{eq:w1}. Therefore, we focus on proving (P3). Notice that it will be enough to show \eqref{eq:rep1a}, as \eqref{eq:rep2a} follows immediately upon taking expectations on \eqref{eq:rep1a} with respect to  $\mathbb{P}$. To show \eqref{eq:rep1a}, we introduce yet another auxiliary random walk, whose law we will denote by $Q^{u}$, given by the transition probabilities
	\begin{equation}\label{eq:w3}
	Q^{u}(X_{n+1}= x + e \,|\, X_n = x) =\frac{ c_{y,\alpha}u(e)}{\alpha(e)}
	\end{equation} for each $e \in \bbV$ and $x \in \Z^d$, where the weights $u(e)$ are the same as before and $c_{y,\alpha}>0$ is a normalizing constant so that the transition probabilities for $Q^{u}$ in \eqref{eq:w3} add up to $1$. As before, we write $Q^{u}_x$ to denote the law of this random walk starting from a fixed $x \in \Z^d$ and use $E^{u}_x$ to denote the expectation with respect to $Q^{u}_x$. 
	
	Having introduced this second auxiliary random walk, the first step will be to show that
	\begin{equation}\label{eq:rep1}
	E^{u}_0\Big(\e^{\langle \theta,X_{n}\rangle}\prod_{j=1}^{n}\omega(X_{j-1},\Delta_j(X))\,;\,X_{n}=x\Big)=(c_{y,\alpha}\sqrt{C_{y,\alpha}})^{n}E_{0,\omega}(\e^{\langle \theta+\theta_{y,\alpha},X_{n}\rangle}\,;\,X_{n}=x),
	\end{equation} for every $\theta \in \R^d$ and $x \in \Z^d$, where $c_{y,\alpha}$ is as in \eqref{eq:w3}, $C_{y,\alpha}$ as in \eqref{eq:w1} and $\theta_{y,\alpha}$ is given by \eqref{eq:deftheta}. To this end, let us define a \textit{path of length $n$} to be any sequence $\bar{x}=(x_0,\dots,x_n)$ of $n + 1$ sites in $\Z^d$ satisfying that $x_j$ and $x_{j-1}$ are nearest neighbors for all $j=1,\dots,n$. Then observe that, for \eqref{eq:rep1} to hold, it is enough to show that 
	\begin{equation}\label{eq:eq}
	Q^{u}_0( (X_0,\dots,X_n) = \bar{x}) = (c_{y,\alpha}\sqrt{C_{y,\alpha}})^{n}\e^{\langle \theta_{y,\alpha}, x\rangle}
	\end{equation} for all paths $\bar{x}$ of length $n$ with $x_0=0$ and $x_n=x$. To check \eqref{eq:eq}, let us fix such a path $\bar{x}$ and denote by $\bar{x}^+_i$ the number of steps made by this path in direction $e_i$ and by $\bar{x}^-_i$ the number of those in direction $-e_i$. Then, since $\bar{x}^+_i = \bar{x}^{-}_i + \langle x,e_i\rangle$, by the Markov property we have that 
	\begin{align*}
	Q^{u,v}_0( (X_0,\dots,X_n) = \bar{x})&=c^n_{y,\mathbb{P}}\prod_{i=1}^d \left(\frac{u(e_i)}{\alpha(e_i)}\right)^{\bar{x}^+_i}\prod_{i=1}^d \left(\frac{u(-e_i)}{\alpha(-e_i)}\right)^{\bar{x}^-_i}\\
	&=c^n_{y,\mathbb{P}}\prod_{i=1}^d \left(\frac{u(e_i)u(-e_i)}{\alpha(e_i)\alpha(-e_i)}\right)^{\bar{x}^-_i}\prod_{i=1}^d \left(\frac{u(e_i)}{\alpha(e_i)}\right)^{\langle x,e_i\rangle}.
	\end{align*}
	Notice that, by construction of the weights $u$, one has that  $\frac{u(e_i)u(-e_i)}{\alpha(e_i)\alpha(-e_i)}=C_{y,\alpha}$ holds. Moreover, from the restriction $\sum_{i=1}^d (\bar{x}^+_i+\bar{x}^-_i)=n$ and the relation $\bar{x}^+_i = \bar{x}^{-}_i + \langle x,e_i\rangle$ for every $i=1,\dots,d$, it follows that $\sum_{i=1}^d \bar{x}^-_i = \tfrac{1}{2}(n-\sum_{i=1}^d \langle x, e_i\rangle)$. Hence, we obtain   
	\begin{equation}\label{eq:eq2}
	Q^{u}_0( (X_0,\dots,X_n) = \bar{x}) =(c_{y,\alpha}\sqrt{C_{y,\alpha}})^{n} \prod_{i=1}^d \left(\frac{u(e_i)}{\alpha(e_i)\sqrt{C_{y,\alpha}}}\right)^{\langle x,e_i\rangle}=(c_{y,\alpha}\sqrt{C_{y,\alpha}})^{n}\e^{\langle \theta_{y,\alpha}, x\rangle}.
	\end{equation} 
	
	Summing \eqref{eq:rep1} over all $x \in \Z^d$ yields
	\begin{equation}\label{eq:rep2}
	E^{u}_0\Big(\e^{\langle \theta,X_{n}\rangle}\prod_{j=1}^{n}\omega(X_{j-1},\Delta_j(X))\Big)=(c_{y,\alpha}\sqrt{C_{y,\alpha}})^{n}E_{0,\omega}(\e^{\langle \theta+\theta_{y,\alpha},X_{n}\rangle}).
	\end{equation} Finally, we conclude \eqref{eq:rep1a} from \eqref{eq:rep2} upon noticing that, by definition of $Q$ and $Q^{u}$, 
	\[
	E^{u}_0\Big(\e^{\langle \theta,X_{n}\rangle} \prod_{j=1}^{n}\omega(X_{j-1},\Delta_j(X)) \Big)=c_{y,\alpha}^{n}E^Q_0\Big(\e^{\langle \theta, X_{n}\rangle}\prod_{j=1}^{n}\xi(X_{j-1},\Delta_j (X))\Big).
	\] This completes the proof.
\end{proof}

As a consequence of Lemma \ref{lemma:prop}, we immediately get the following corollary.

\begin{cor}\label{cor:exist} For $\theta \in \R^d$, the quantities
	\begin{equation}\label{eq:mgf2}
	\overline{\Lambda}_{q}(\theta):=\lim_{n\to \infty}\frac{1}{n}\log E^Q_0\Big(\e^{\langle \theta, X_{n}\rangle} \prod_{j=1}^{n}\xi(X_{j-1},\Delta_j(X))\Big),
	\end{equation} and
	\begin{equation}\label{eq:mgf1}
	\overline{\Lambda}_{a}(\theta):=\lim_{n\to \infty}\frac{1}{n}\log E^Q_0\Big(\e^{\langle \theta, X_{n}\rangle} \E\prod_{j=1}^{n}\xi(X_{j-1},\Delta_j(X))\Big)
	\end{equation}
	are well-defined, i.e. the limits in \eqref{eq:mgf2} and \eqref{eq:mgf1} both exist, are finite and the right-hand side of \eqref{eq:mgf2} is $\mathbb{P}$-almost surely constant. 	
\end{cor}

\begin{proof} It follows from \eqref{eq:rep1a} and \eqref{eq:rep2a} that, for any $\theta \in \R^d$,
	\begin{equation}\label{eq:eqmgf}
	\overline{\Lambda}_{q}(\theta)= \log(\sqrt{C_{y,\alpha}}) + \Lambda_q(\theta+\theta_{y,\alpha}) \hspace{1cm}\text{ and }\hspace{1cm}\overline{\Lambda}_{a}(\theta)= \log(\sqrt{C_{y,\alpha}}) + \Lambda_a(\theta+\theta_{y,\alpha}), 
	\end{equation} where $\Lambda_{q}(\theta):=\lim_{n\to\infty}\frac{1}{n}\log E_{0,\omega}(\e^{\langle \theta,X_{n}\rangle})$ and $\Lambda_{a}(\theta):=\lim_{n\to\infty}\frac{1}{n}\log E_{0}(\e^{\langle \theta,X_{n}\rangle})$ respectively denote the quenched and annealed limiting logarithmic moment generating functions associated with the RWRE. Since both $\Lambda_q$ and $\Lambda_a$ are well-defined in the sense described in the statement of Corollary \ref{cor:exist} (see \cite[Theorem 2.6]{RS14} for the quenched case and, in the annealed case, this follows from \cite[Theorem 3.2]{V03} and Varadhan's Lemma \cite[Theorem 4.3.1]{DZ98}), we see that $\overline{\Lambda}_{q}(\theta)$ and $\overline{\Lambda}_{a}(\theta)$ are so as well.	
\end{proof}

The following remark contains some crucial estimates that we will use extensively in the sequel. 

\begin{remark}\label{rem:bound} Given any $\theta,\theta' \in \R^d$ and environmental law $\mathbb{P}$, for any $n \geq 1$ we have 
	\[
	|\overline{\Lambda}_{a}(\theta)-\overline{\Lambda}_{a}(\theta')|\leq |\theta-\theta'| \hspace{1cm}\text{ and }\hspace{1cm} \e^{-\overline{\Lambda}_{a}(0)n}\prod_{j=1}^n \xi(X_{j-1},\Delta_j(X)) \leq \e^{h(\mathrm{dis}(\P))n}\quad \P\text{-a.s.},
	\] where $h(x):=\log \big(\frac{1+x}{1-x}\big)$ for $x \in [0,1)$. The proof of these inequalities is elementary, so we omit it. Nevertheless, from now onwards we will assume that $\mathrm{dis}(\P) < 1$ so that the expression $h(\mathrm{dis}(\P))<1$, which will appear numerous times in the sequel, is always well-defined. This does not represent any real loss of generality since we shall always be interested in environmental laws with small enough disorder.
\end{remark}

The main objective in Section \ref{sec:auxiliary} is to show that $\overline{\Lambda}_{a}(\theta)=\overline{\Lambda}_{q}(\theta)$ for $\theta$ close enough to $0$, whenever the disorder of the environment is sufficiently low. We will later see in Section \ref{sec:Hessian} that, in turn, this will imply that $I_{q}(x)=I_{a}(x)$ for $x$ sufficiently close to $y$. To carry out all this, we shall exploit a renewal structure available for the $Q$-random walk. We introduce this renewal structure next.

\subsection{A renewal structure for the $Q$-random walk}\label{sec:renewal} 
Let us first fix a direction $\ell \in \bbV$ such that $E^Q_0(\langle X_1 , \ell\rangle) > 0$. Notice that such a direction always exists since $E^Q_0(\langle X_1 , \ell\rangle)=\langle y,\ell\rangle$ by Lemma \ref{lemma:prop} and $y \neq 0$ by assumption. We then set for $u \in \R$, 
\begin{align*}
H_{u}:&=\inf\{n\geq 1: \langle X_{n},\ell \rangle>u \},\,\,\,
S_0:= 0,\,\,\,\,\,
\beta_0:=\inf\{n \geq 1: \langle X_n,\ell\rangle <\langle X_{0},\ell \rangle\},\,\,\,
R_0:= \langle X_0, \ell\rangle
\end{align*} and define the sequences of stopping times $(S_k)_{k \in \N_0},(\beta_k)_{k \in \N_0}$ and $(R_k)_{k \in \N_0}$ inductively as 
\begin{align*}
S_{k+1}&:=H_{R_{k}},\,\,\, 
\beta_{k+1}:=\inf\{n > S_{k+1} : \langle X_n,\ell\rangle <\langle X_{S_{k+1}},\ell \rangle\},\\
R_{k+1}&:=\begin{cases}\sup\{\langle X_{n},\ell\rangle: 0\leq n\leq \beta_{k+1}\} & \text{ if }\beta_{k+1}<\infty\\ \langle X_{S_{k+1}},\ell\rangle &\text{ if }\beta_{k+1}=\infty,\end{cases}
\end{align*} with the convention that $\inf \emptyset = \infty$. Observe that, by choice of $\ell$ and the law of large numbers, we have $\lim_{n \to \infty} \langle X_n , \ell \rangle = \infty$ $Q$-almost surely. In particular, this implies that 
\[
R_k < \infty \,\, Q\text{-a.s.} \Longrightarrow S_{k+1} < \infty \,\, Q\text{-a.s.} \Longrightarrow R_{k+1} < \infty \,\, Q\text{-a.s.}, 
\] so that by induction all $S_k$ and $R_k$ are finite $Q$-almost surely. However, the $\beta_k$ will not all be. Thus, we define the sequence $(\tau_k)_{k \in \N_0}$ of \textit{renewal times} as 
\[
\tau_k:=S_{W_k},
\] where $(W_k)_{k \in \N_0}$ is defined inductively by first taking $W_0:=0$ and then setting
\[
W_{k+1}:= \inf \{ n > W_k : \beta_n = \infty\}.
\] That the renewal times $\tau_k$ are well-defined is a consequence of the fact that all $W_k$ are $Q$-a.s. finite, which in turn follows from the Markov property and Lemma \ref{lemma:pos} below. 

\begin{lemma} \label{lemma:pos} There exists $\overline{c}=\overline{c}(y) > 0$ such that $Q(\beta_0=\infty)>\overline{c}$ for any $\P \in \mathcal{P}_\kappa$, where $Q=Q(y,\alpha)$ is the law of the $Q(y,\alpha)$-random walk with jump weights given by \eqref{eq:w1}.
\end{lemma}

\begin{proof} Let $(Z_n)_{n \in \N_0}$ be the random walk on $\Z$ which starts from $\langle X_0, \ell\rangle$ and, at each step $n \in \N_0$, jumps one unit to the left with probability $q$ and one to the right with probability $p:=1-q$, where
	\[
	q:=-\frac{\langle y,\ell \rangle}{2}+\frac{1}{2}\sqrt{|\langle  y,\ell\rangle|^{2}+1}. 
	\] Observe that, since $f(1/4\alpha(\ell)\alpha(-\ell)) > 1$ where $f$ is as in \eqref{def:f}, we have that $C_{y,\alpha}< \tfrac{1}{4\alpha(\ell)\alpha(-\ell)}$ and thus $u(\ell) \leq q$.  It follows that we may couple $(Z_n)_{n \in \N_0}$ with our $Q$-random walk in such a way that, for all $n \in \N_0$,
	\[
	Z_n \leq \langle X_n,\ell \rangle \Longrightarrow  Z_{n+1} \leq \langle X_{n+1},\ell \rangle.
	\] In particular, if we denote $\beta_0(Z):=\inf \{n \geq 1 : Z_n < Z_0\}$ then $P(\beta_0(Z) = \infty) \leq Q(\beta_0 =\infty)$. But, since $q < \tfrac{1}{2}$ by Minkowski's inequality, by standard gambler's ruin estimates we have
	\[P(\beta_0(Z) = \infty)= 1 - \tfrac{q}{p} =:\overline{c}.
	\] This concludes the proof.
\end{proof}

It follows from this construction above that all renewal times $\tau_k$ are $Q$-a.s. finite, that $(X_{\tau_{1}},\tau_{1})$ is independent of the sequence $(X_{\tau_{k+1}}-X_{\tau_{k}},\tau_{k+1}-\tau_{k})_{k \geq 1}$ and that this last sequence is i.i.d. with common law given by that of  $(X_{\tau_{1}},\tau_{1})$ conditioned on the event $\{\beta_0=\infty \}$. We now investigate some (uniform in $\P$) integrability properties of these renewal times.


\begin{lemma}\label{lemma2} There exists $\rho_1=\rho_1(y) > 0$ such that $E^Q_0(\e^{\rho \langle X_{\tau_{1}},\ell \rangle })\leq \frac{3}{\overline{c}}$ for all $\rho < \rho_1$ and any $\P \in \mathcal{P}_\kappa$, where $\overline{c}$ is the constant from Lemma \ref{lemma:pos}.
\end{lemma}
\begin{proof} By splitting $E^Q_0(\e^{\rho \langle X_{\tau_{1}},\ell\rangle })$ according to the value of $W_1$, we obtain the bound
	\begin{equation}\label{eq:d1}
	E^Q_0(\e^{\rho \langle X_{\tau_{1}},\ell\rangle }) = \sum_{k=1}^\infty E^Q_0(\e^{\rho \langle X_{\tau_{1}},\ell\rangle }\,;\, W_1 = k) \leq \sum_{k=1}^\infty E^Q_0(\e^{\rho \langle X_{S_k},\ell\rangle }\,;\, \beta_j < \infty \text{ for }j=1,\dots,k). 
	\end{equation} Observe that $|\langle X_{S_1}, \ell\rangle| \leq 1 + |X_0|$  by definition of $R_0$ and the fact that the walk is nearest neighbor, so that the first term in the sum on the right-hand side of \eqref{eq:d1} is bounded from above by $\e^\rho$. 
	
	On the other hand, since $R_{k-1}= \sup \{ \langle X_n , \ell \rangle : S_{k-1} \leq n \leq \beta_{k-1}\}$ when $\beta_{k-1}< \infty$ and $k \geq 2$, by writing $\langle X_{S_k},\ell\rangle = \langle X_{S_{k-1}},\ell\rangle + \langle X_{S_{k}}-X_{S_{k-1}},\ell\rangle$ and using the Markov property at time $S_{k-1}$, we see that for $k \geq 2$ the $k$-th term in the right-hand side of \eqref{eq:d1} is bounded from above by
	\[
	E^Q_0(\e^{\rho \langle X_{S_{k-1}},\ell\rangle }\,;\, \beta_j < \infty \text{ for }j=1,\dots,k-1) E^Q_0( \e^{\rho(1+\mathcal{R})} \,;\,\beta_0 < \infty), 
	\]
	where $\mathcal{R}:=\sup \{\langle X_{n},\ell \rangle : 0 \leq n \leq \beta_0 \}$. Repeating this argument all the way down to $\langle X_{S_1},\ell\rangle$ and then using the bound for the case $k=1$ yields the bound 
	\begin{equation} \label{eq:d1star}
	E^Q_0(\e^{\rho \langle X_{\tau_{1}},\ell\rangle }) \leq \e^{\rho}\sum_{k=1}^\infty \left( E^Q_0\left(\e^{\rho (1+\mathcal{R})}\,;\, \beta_0<\infty \right)\right)^{k-1}.
	\end{equation}
	Therefore, in order to complete the proof we only need to show that, for $\rho$ small enough depending only on $y$, we have 
	\begin{equation}\label{eq:b1}
	E^Q_0\left(\e^{\rho (1+\mathcal{R})}\,;\,\beta_0<\infty \right)<1-\frac{\overline{c}}{2}.
	\end{equation} But, by the union bound and Lemma \ref{lemma:pos}, for any $N \geq 1$ the expectation on the left-hand side of \eqref{eq:b1} is bounded from above by
	\begin{align*}
	\e^{\rho N}Q_0(\beta_0<\infty) &+ \sum_{n=N}^\infty \e^{\rho(2+n)}Q_0( n \leq \mathcal{R}< n+1\,,\,\beta_0 <\infty)\\& \leq \e^{\rho N}(1-\overline{c}) +\sum_{n=N}^\infty \e^{\rho(2+n)}Q_0( \mathcal{R}\geq n\,,\,\beta_0 <\infty).
	\end{align*} Now, observe that for $n \geq 1$
	\begin{align*}
	Q_0(\mathcal{R}\geq n \,,\, \beta_0<\infty) &\leq Q_0(n\leq \beta_0<\infty)\\
	&=\sum_{k=n}^{\infty}Q_0(\beta_0=i)\leq \sum_{k=n}^{\infty}Q_0(\langle X_{k},\ell\rangle <0)\leq \frac{\e^{-\tfrac{1}{8}|\langle y,\ell\rangle|^2 n}}{1-\e^{-\tfrac{1}{8}|\langle y,\ell\rangle|^2}},
	\end{align*} where to obtain the last inequality we have used the bound $Q_0(\langle X_{k},\ell\rangle <0) \leq \e^{-\tfrac{1}{8}|\langle y,\ell\rangle|^2k}$, which follows from the (one-sided) Azuma-Hoeffding inequality for the martingale $(M_n)_{n \in \N_0}$ given by $M_n:= \langle X_n ,\ell \rangle - n\langle y,\ell\rangle$ (whose increments are bounded by $2$). Thus, we see that, for any $N \geq 1$, 
	\[
	E^Q_0\left(\e^{\rho (1+\mathcal{R})}\,;\,\beta_0<\infty \right) \leq \e^{\rho(N \wedge 2)}\Bigg(1-\overline{c} +\frac{\e^{-\tfrac{1}{8}|\langle y,\ell\rangle|^2N}}{\big(1-\e^{-\tfrac{1}{8}|\langle y,\ell\rangle|^2}\big)^2}\Bigg)
	\] from where \eqref{eq:b1} now follows by taking first $N$ sufficiently large and then $\rho$ accordingly small.
\end{proof}

As a consequence of Lemma \ref{lemma2}, we obtain (uniform in $\P$) exponential moments for $\tau_1$.

\begin{prop}\label{prop:exp_moments} There exists $\gamma_{0}=\gamma_0(y) > 0$ such that $E^Q_0(\e^{\gamma \tau_1}) \leq 2$ for all $\gamma \leq  \gamma_{0}$ and any $\P \in \mathcal{P}_\kappa$. 
\end{prop}

\begin{proof} For $n \geq 1$, by the union bound we have
	\[
	Q_0(\tau_{1}>n) \leq Q_0\left(\langle X_{\tau_{1}},\ell \rangle >\frac{\langle y, \ell\rangle}{2}\right)+Q_0\left(\tau_{1}>n\,,\,\langle X_{\tau_{1}},\ell\rangle \leq \frac{\langle y, \ell\rangle n}{2}\right).
	\]
	Using the exponential Tchebychev inequality and Lemma \ref{lemma2}, we have 
	\[
	Q_0\left(\langle X_{\tau_{1}},\ell \rangle >\frac{\langle y, \ell\rangle}{2}\right) \leq \e^{-\rho\langle y,\ell\rangle n} E^Q_0( \e^{2\rho\langle X_{\tau_1},\ell\rangle}) \leq C\e^{-\rho\langle y,\ell\rangle n} 
	\] for some $C,\rho > 0$ depending only on $y$. On the other hand, by definition of $\tau_1$ we have \begin{align*}
	Q_0\left(\tau_{1}>n\,,\,\langle X_{\tau_{1}},\ell\rangle \leq \frac{\langle y, \ell\rangle n}{2}\right)\leq Q_0\left(\langle X_{n},\ell\rangle \leq \frac{\langle y, \ell \rangle n}{2} \right) \leq  \e^{-\tfrac{1}{32}|\langle y , \ell\rangle|^2n},
	\end{align*} where to obtain the last inequality we have used the (one-sided) Azuma-Hoeffding inequality for the martingale $(M_n)_{n \in \N_0}$ as in the proof of Lemma \ref{lemma2}. Hence, we see that there exist $C,\gamma> 0$ depending only on $y$ such that $Q_0(\tau_{1}>n)\leq C\e^{-\gamma n}$ for all $n \geq 1$. From this the result now follows by an argument similar to the one used to derive \eqref{eq:b1}. 
\end{proof}

Finally, the above regeneration structure, together with Remark \ref{rem:bound}, allows us to deduce analyticity of $\overline{\Lambda}_a$.

\begin{prop}\label{prop:reglambdaa} There exists $\gamma_1>0$  (determined by Proposition \ref{prop:norm1} below), 
such that if $\mathrm{dis}(\P) < \gamma_1$ then the mapping $\theta \mapsto \overline{\Lambda}_a(\theta)$ is analytic on the set $\{\theta:|\theta|<\gamma_1\}$.
\end{prop}

\begin{proof} 
Consider the function $\Psi:\R^d \times \R \to \R$ defined as
	\[
	\Psi(\theta,r):=\overline{E}^{Q}_0\Big(\e^{\langle \theta, X_{\tau_1}\rangle -r\tau_1}\E\prod_{j=1}^{\tau_1}\xi(X_{j-1},\Delta_j(X))\Big)
	\] where $\overline{E}^Q_0$ above stands for expectation with respect to $\overline{Q}_0$, the law $Q_0$ conditioned on the event $\{\beta_0 = \infty\}$. By Remark \ref{rem:bound} we have that, whenever $r=\overline{\Lambda}_a(\theta)+\delta$ for some $\delta \in \R$, 
	\[
	\Big|\langle \theta, X_{\tau_1}\rangle -r\tau_1 + \log\E\prod_{j=1}^{\tau_1}\xi(X_{j-1},\Delta_j(X))\Big| \leq (2|\theta|+h(\mathrm{dis}(\P))+|\delta|)\tau_1
	\] so that, by choice of $\gamma_1$ (see the proof of Lemma \ref{lemma3} for details), we have
	\begin{equation}\label{eq:cotaanalitica}
	\overline{E}^{Q}_0 \Big(\tau_1 \exp\Big\{|\langle \theta, X_{\tau_1}\rangle -(\overline{\Lambda}_a(\theta)+\delta)\tau_1 + \log\E\prod_{j=1}^{\tau_1}\xi(X_{j-1},\Delta_j(X))|\Big\}\Big) < \infty
	\end{equation} whenever $|\theta| \vee \mathrm{dis}(\P) < \gamma_1$ and $|\delta|< \delta_c $ for some $\delta_c=\delta_c(y) > 0$ small enough. 
	It then follows from \eqref{eq:cotaanalitica}, dominated convergence and Remark \ref{rem:bound} once again that, when $\mathrm{dis}(\P)<\gamma_1$, $\Psi$ is analytic on the open set $\mathcal{C}_y:=\{ (\theta,r): |\theta|< \gamma_1\,,\,|r-\overline{\Lambda}_a(\theta)|<\delta_c\}$ with series expansion given by 
	\[
	\Psi(\theta,r)= \sum_{n=0}^\infty \frac{\overline{E}^Q_0( (\langle \theta, X_{\tau_1}\rangle -r\tau_1)^n)}{n!}
	\] and $\partial_r \Psi$ given by 
	\begin{equation}\label{eq:cotaderiv}
	\partial_r\Psi(\theta,r)=-\overline{E}^{Q}_0\Big(\tau_1\e^{\langle \theta, X_{\tau_1}\rangle -r\tau_1}\E\prod_{j=1}^{\tau_1}\xi(X_{j-1},\Delta_j(X))\Big).
	\end{equation} But observe that $\Psi(\theta,\overline{\Lambda}_a(\theta))=1$ whenever $|\theta| \vee \mathrm{dis}(\P) < \gamma_1$ by Proposition \ref{prop:norm1}, which in turn implies that $-\partial_r\Psi(\theta,\overline{\Lambda}_a(\theta)) \geq \Psi(\theta,\overline{\Lambda}_a(\theta))=1>0$ by \eqref{eq:cotaderiv}. Therefore, the analyticity of $\overline{\Lambda}_a(\theta)$ for $|\theta|< \gamma_1$ whenever $\mathrm{dis}(\P) < \gamma_1$ now follows from the analytic implicit function theorem, see \cite[Theorem 6.1.2]{KP02}.  
\end{proof}

\subsection{Equality of $\overline{\Lambda}_{q}$ and $\overline{\Lambda}_{a}$: the main argument} We now describe the main steps in the proof of the equality of $\overline{\Lambda}_{a}(\theta)$ and $\overline{\Lambda}_{q}(\theta)$ for $\theta$ close enough to $0$, whenever the disorder of the environment is sufficiently low. 
The more technical details are deferred to a separate section. We begin by introducing the key object in our analysis. 

\begin{definition} Given $n \geq 1$, $\theta \in \R^d$ and an environment $\omega$, we define
	\begin{equation}\label{gn_def}
	\Phi_{n}(\theta,\omega):=\overline{E}^Q_0\Big(\e^{\langle \theta, X_{L_{n}}\rangle-\overline{\Lambda}_{a}(\theta)L_{n}}\prod_{j=1}^{L_{n}}\xi(X_{j-1},\Delta_j(X)),L_{n}=\tau_{k}\text{ for some } k \geq 1\Big),
	\end{equation} where, as before, $\overline{E}^Q_0$ above stands for expectation with respect to $\overline{Q}_0$, the law $Q_0$ conditioned on the event $\{\beta_0 = \infty\}$, and $L_{n}:=\inf\{n\geq 1: \langle X_{n}-X_{0},\ell\rangle=n\}$. Throughout the sequel we shall write $\Phi_n(\theta)$ instead of $\Phi_n(\theta,\omega)$ whenever we think of $\omega$ as being random (and therefore of $\Phi_n(\theta)$ as being a random variable).
\end{definition} The following two propositions contain the crucial information about the random variable $\Phi_n$. 

\begin{prop}\label{prop:norm1} There exists $\gamma_1=\gamma_1(y)>0$ such that, for any $\P \in \mathcal{P}_\kappa$, whenever $|\theta| \vee \mathrm{dis}(\mathbb{P}) < \gamma_1$ we have 
\begin{equation}\label{eq:m1}
\overline{E}^Q_0\Big(\e^{\langle \theta,X_{\tau_{1}}\rangle -\overline{\Lambda}_{a}(\theta) \tau_{1}}\E\prod_{j=1}^{\tau_{1}}\xi(X_{j-1},\Delta_j(X))\Big)=1
\end{equation} and 
\[\lim_{n\to\infty}\E \Phi_n(\theta) > 0.
\] 
\end{prop} 

\begin{prop}\label{prop:norm2} There exists $\gamma_{2}=\gamma_2(y,d,\kappa) > 0$ such that, for any $\P \in \mathcal{P}_\kappa$, whenever $|\theta| \vee \mathrm{dis}(\mathbb{P}) < \gamma_2$ we have 
	$
	\sup_{n\geq 1}\E (\Phi_{n}(\theta))^{2}<\infty.
	$
\end{prop}

The proofs of these propositions are deferred to Section \ref{sec:l12}. Let us first conclude 

\noindent{\bf Proof of $\overline\Lambda_q=\overline\Lambda_a$ (assuming Proposition \ref{prop:norm1} and Proposition \ref{prop:norm2}):}  Note that by Propositions \ref{prop:norm1}-\ref{prop:norm2}, whenever $|\theta|\vee \mathrm{dis}(\mathbb{P}) < \gamma_1 \wedge \gamma_2$ we have
\begin{equation}\label{eq:ineq1}
\mathbb{P}\Big( \lim_{n \to \infty} \Phi_n(\theta) = 0\Big) < 1.
\end{equation} Indeed, if $\Phi_n(\theta) \to 0$ $\mathbb{P}$-a.s. then $\lim_{n \to \infty}\E \Phi_{n}(\theta)= 0$ since $(\Phi_n(\theta))_{n \geq 1}$ is uniformly integrable by Proposition \ref{prop:norm2}. However, this is in contradiction with Proposition \ref{prop:norm1} and thus \eqref{eq:ineq1} must hold. Furthermore, we also have the following.

\begin{lemma}\label{lemma7}
	For any $\theta\in \R^{d}$ and $\delta>0$, we have
	\begin{equation}\label{eq:ineq2}
	\mathbb{P}\left(\lim_{n\to \infty}E^Q_0\Big(\e^{\langle \theta, X_{L_{n}}\rangle-(\overline{\Lambda}_{q}(\theta)+\delta) L_{n}}\prod_{j=1}^{L_{n}}\xi(X_{j-1},\Delta_j(X))\Big)=0\right)=1.
	\end{equation}
\end{lemma}
\begin{proof} Let us write $\lambda_{\theta,\delta}:=\overline{\Lambda}_{q}(\theta)+\delta$ in the sequel for simplicity. Then, by splitting the expectation on the left-hand side of \eqref{eq:ineq2} according to the different possible values for $L_n$, we can bound it from above by 
	\begin{equation}\label{eq:bound1}
	\sum_{k=n}^{\infty}E^Q_0\Big(\e^{\langle \theta, X_{k}\rangle-\lambda_{\theta,\delta} k}\prod_{j=1}^{k}\xi(X_{j-1},\Delta_j(X))\Big)=\sum_{k=n}^{\infty}\e^{-\lambda_{\theta,\delta}k}E^Q_0\Big(\e^{\langle \theta, X_{k}\rangle}\prod_{j=1}^{k}\xi(X_{j-1},\Delta_j(X))\Big).
	\end{equation} Now, since for $\mathbb{P}$-almost every $\omega$ we have
	\[
	E^Q_0\Big(\e^{\langle \theta, X_{k}\rangle}\prod_{j=1}^{k}\xi(X_{j-1},\Delta_j(X))\Big) = \e^{(\overline{\Lambda}_{q}(\theta)+o_\omega(1))k}
	\] for some $o_\omega(1) \to 0$ as $k \to \infty$, from \eqref{eq:bound1} we obtain that for all $n$ sufficiently large and $\mathbb{P}$-a.e. $\omega$,
	\[
	E^Q_0\Big(\e^{\langle \theta, X_{L_{n}}\rangle-(\overline{\Lambda}_{q}(\theta)+\delta) L_{n}}\prod_{j=1}^{L_{n}}\xi(X_{i-1},\Delta_j(X))\Big) \leq \sum_{k=n}^\infty \e^{-\frac{\delta}{2}k} \leq \frac{\e^{-\frac{\delta}{2}n}}{1-\e^{-\delta/2}}.
	\] Taking $n\to\infty$ on this inequality now allows us to conclude.
\end{proof}

%
Combined with \eqref{eq:ineq1}, Lemma \ref{lemma7} yields the equality  $\overline{\Lambda}_{a}(\theta)=\overline{\Lambda}_{q}(\theta)$ whenever $|\theta|\vee \mathrm{dis}(\mathbb{P}) < \gamma_1 \wedge \gamma_2$. We state and prove this in a separate proposition for future reference.

\begin{prop}\label{prop1} Define $\overline{\gamma}=\gamma_1 \wedge \gamma_2$, for $\gamma_1$ and $\gamma_2$ as in Propositions \ref{prop:norm1} and \ref{prop:norm2}, respectively. Then, for any $\P \in \mathcal{P}_\kappa$, whenever $|\theta|\vee \mathrm{dis}(\mathbb{P}) < \overline{\gamma}$ we have $\overline{\Lambda}_{q}(\theta)=\overline{\Lambda}_{a}(\theta)$.
\end{prop}
\begin{proof} Observe that \eqref{eq:ineq1} implies that, for $|\theta|\vee\mathrm{dis}(\mathbb{P}) < \overline{\gamma}$, 
	\[
	\mathbb{P}\left(\limsup_{n\to \infty}E^Q_0\Big(\e^{\langle \theta, X_{L_{n}}\rangle-(\overline{\Lambda}_{a}(\theta)) L_{n}}\prod_{j=1}^{L_{n}}\xi(X_{j-1},\Delta_j(X))\Big)>0\right)>0.
	\] In conjunction with \eqref{eq:ineq2}, this yields the existence of an environment $\omega$ and $n \geq 1$ such that
	\[
	E^Q_0\Big(\e^{\langle \theta, X_{L_{n}}\rangle-(\overline{\Lambda}_{q}(\theta)+\delta) L_{n}}\prod_{j=1}^{L_{n}}\xi(X_{i-1},\Delta_j(X))\Big) < E^Q_0\Big(\e^{\langle \theta, X_{L_{n}}\rangle-(\overline{\Lambda}_{a}(\theta)) L_{n}}\prod_{j=1}^{L_{n}}\xi(X_{j-1},\Delta_j(X))\Big),
	\] from where it follows that $\overline{\Lambda}_{q}(\theta)+\delta > \overline{\Lambda}_{a}(\theta)$. Letting $\delta \to 0$ yields the inequality $\overline{\Lambda}_{q}(\theta)\geq \overline{\Lambda}_{a}(\theta)$. But, since $\overline{\Lambda}_{q}(\theta)\leq \overline{\Lambda}_{a}(\theta)$ for all $\theta \in \R^{d}$ by Jensen's inequality, we deduce that $\overline{\Lambda}_{q}(\theta)=\overline{\Lambda}_{a}(\theta)$ whenever $|\theta|\vee\mathrm{dis}(\mathbb{P}) < \overline{\gamma}$, which concludes the proof.
\end{proof}

Thus, to complete the argument it only remains to prove Propositions \ref{prop:norm1} and \ref{prop:norm2}. We will do this later in Section \ref{sec:l12}.

\section{Proof of Theorem \ref{th:main} and Theorem \ref{th:main2}: Deducing $I_q=I_a$ from $\overline{\Lambda}_q=\overline{\Lambda}_a$} \label{sec:Hessian}

We now show how to conclude Theorem \ref{th:main2} (and therefore, Theorem \ref{th:main}) from the results in the previous section by proving that the equality of $\overline{\Lambda}_q$ and $\overline{\Lambda}_a$ in a neighborhood of the origin implies, for sufficiently small disorder, the equality of the rate functions $I_q$ and $I_a$ in a neighborhood of $y$. The task will be carried out in three steps, spanning Section \ref{sec:uniform}-Section \ref{sec:proof:prop2}.

\subsection{Uniform closeness of $y$ and $\nabla\overline\Lambda_a(0)$.}\label{sec:uniform}
As already remarked earlier, we would like to argue that, given $y\ne 0$, for all environmental laws with a small enough disorder, $y$ is close to the gradient $\nabla\overline\Lambda_a(0)$. 
 Recall that by Proposition \ref{prop:norm1} we have that, for any $\P \in \mathcal{P}_\kappa$, if $|\theta| \vee \mathrm{dis}(\mathbb{P}) < \gamma_1$ then 
\[
\overline{E}^Q_0\Big(\e^{\langle \theta,X_{\tau_{1}}\rangle -\overline{\Lambda}_{a}(\theta) \tau_{1}}\E\prod_{j=1}^{\tau_{1}}\xi(X_{j-1},\Delta_j(X))\Big)=1.
\] In particular, taking gradient on both sides (which we can do by dominated convergence, using Proposition \ref{prop:reglambdaa} and the control in \eqref{eq:boundsup}), we obtain that whenever $|\theta| \vee \mathrm{dis}(\mathbb{P}) < \gamma_1$,
\begin{equation}\label{eq:gradient_identity}
\overline{E}^Q_0\Big((X_{\tau_{1}}-\nabla \overline{\Lambda}_{a}(\theta)\tau_{1})\e^{\langle \theta,X_{\tau_{1}}\rangle -\overline{\Lambda}_{a}(\theta) \tau_{1}}\E\prod_{j=1}^{\tau_{1}}\xi(X_{j-1},\Delta_j(X))\Big)=0,
\end{equation}
which yields the representation
\begin{equation}\label{eq:rep}
\nabla \overline{\Lambda}_{a}(\theta)=\frac{\overline{E}^Q_0\Big(X_{\tau_{1}}\e^{\langle \theta,X_{\tau_{1}}\rangle -\overline{\Lambda}_{a}(\theta) \tau_{1}}\E\prod_{j=1}^{\tau_{1}}\xi(X_{j-1},\Delta_j(X))\Big)}{\overline{E}^Q_0\Big(\tau_{1}\e^{\langle \theta,X_{\tau_{1}}\rangle -\overline{\Lambda}_{a}(\theta) \tau_{1}}\E\prod_{j=1}^{\tau_{1}}\xi(X_{j-1},\Delta_j(X))\Big)}.
\end{equation}
In particular, notice that whenever $\mathrm{dis}(\P)=0$, i.e.  $\P$-a.s. $\omega(x,e)=\alpha(e)$ for all $e \in \bbV$ and $x \in \Z^d$, we have $\overline{\Lambda}_{a}(0)=0$ so that
\begin{equation}\label{eq:dis0}
\nabla\overline{\Lambda}_{a}(0)=\frac{\overline{E}^Q_0(X_{\tau_{1}})}{\overline{E}^Q_0(\tau_1)}.
\end{equation} On the other hand, by the renewal structure, the law of large numbers for the $Q$-random walk and (P2) in Lemma \ref{lemma:prop} we have that, for any environmental law $\P$ (with not necessarily zero disorder),
\begin{equation}\label{eq:dis1}
\frac{\overline{E}^Q_0(X_{\tau_{1}})}{\overline{E}^Q_0(\tau_1)}=y.
\end{equation} In particular, in the zero disorder case we conclude that $\nabla \overline{\Lambda}_a(0)=y$. In the general case, whenever $\mathrm{dis}(\P)$ is sufficiently small $\nabla \overline{\Lambda}_a(0)$ will be close to $y$. More precisely, we have the following.

\begin{prop}\label{lemma9} Given $\delta > 0$ there exists $\eps_1=\eps_1(y,\delta) >0$ such that, for any $\P \in \mathcal{P}_\kappa$, if $\mathrm{dis}(\mathbb{P})<\eps_1$ then $|\nabla \overline{\Lambda}_{a}(0)-y|<\delta$.
\end{prop}
\begin{proof} It follows from \eqref{eq:rep} that
	\[
	\nabla \overline{\Lambda}_{a}(0)=\frac{\overline{E}^Q_0\Big(X_{\tau_{1}}\e^{-\overline{\Lambda}_{a}(0) \tau_{1}}\E\prod_{j=1}^{\tau_{1}}\xi(X_{j-1},\Delta_j(X))\Big)}{\overline{E}^Q_0\Big(\tau_{1}\e^{-\overline{\Lambda}_{a}(0) \tau_{1}}\E\prod_{j=1}^{\tau_{1}}\xi(X_{j-1},\Delta_j(X))\Big)}.
	\] Thus, in light of \eqref{eq:dis1} and since $\overline{E}^Q_0(\tau_1) \geq 1$, in order to prove the result it will suffice to show that given $\delta'>0$ there exists $\varepsilon'_1=\varepsilon'_1(y,\delta')>0$ such that, for any $\mathbb{P} \in \mathcal{P}_\kappa$, if $\mathrm{dis}(\mathbb{P})< \varepsilon'_1$ then
	\begin{equation}\label{eq:bexp1}
	\Big|\overline{E}^Q_0\Big(X_{\tau_{1}}\e^{-\overline{\Lambda}_{a}(0) \tau_{1}}\E\prod_{j=1}^{\tau_{1}}\xi(X_{j-1},\Delta_j(X))\Big) - \overline{E}^Q_0(X_{\tau_1})\Big|\leq\delta'
	\end{equation} and
	\begin{equation}\label{eq:bexp2}
	\Big|\overline{E}^Q_0\Big(\tau_{1}\e^{-\overline{\Lambda}_{a}(0) \tau_{1}}\E\prod_{j=1}^{\tau_{1}}\xi(X_{j-1},\Delta_j(X))\Big) - \overline{E}^Q_0(\tau_1)\Big|\leq \delta'.
	\end{equation} But by Remark \ref{rem:bound} and the the mean value theorem we have that
	\[
	\Big|\overline{E}^Q_0\Big(X_{\tau_{1}}\e^{-\overline{\Lambda}_{a}(0) \tau_{1}}\E\prod_{j=1}^{\tau_{1}}\xi(X_{j-1},\Delta_j(X))\Big) - \overline{E}^Q_0(X_{\tau_1})\Big|\leq  h(\mathrm{dis}(\P)) \overline{E}^Q_0\Big(|X_{\tau_1}|\tau_1 \e^{h(\mathrm{dis}(\P))\tau_1}\Big),
	\] so that \eqref{eq:bexp1} now follows from the bound $|X_{\tau_1}| \leq \tau_1$, Lemma \ref{lemma:pos} and Proposition \ref{prop:exp_moments} upon taking $\mathrm{dis}(\P)$ small enough (depending only on $y$ and $\delta'$). Since \eqref{eq:bexp2} also follows in a similar way, this concludes the proof.
\end{proof}

Next, we consider the set
\[
\mathcal{A}_{y,\P}:=\{\nabla\overline{\Lambda}_{a}(\theta):\abs{\theta}<\overline{\gamma}\},
\] with $\overline{\gamma}$ as in Proposition \ref{prop1}. Observe that this set depends on both $y$ and $\P$ (and we stress this dependence in the notation). The next proposition shows that this set is open when $\mathrm{dis}(\mathbb{P}) < \gamma_1$.

\begin{prop}\label{prop:hessian} For any $\P \in \mathcal{P}_\kappa$, whenever $|\theta| \vee \mathrm{dis}(\mathbb{P}) < \gamma_1$, with $\gamma_1>0$ given by Proposition \ref{prop:norm1}, the Hessian $H_a(\theta)$ of $\overline{\Lambda}_a$ at the point $\theta$ is given by the formula
	\begin{equation}\label{eq:hessian_formula}
	H_a(\theta) =\frac{\overline{E}^Q_0\Big((X_{\tau_{1}}-\nabla\overline{\Lambda}_{a}(\theta)\tau_{1})^T(X_{\tau_{1}}-\nabla\overline{\Lambda}_{a}(\theta)\tau_{1})\e^{\langle \theta,X_{\tau_{1}}\rangle -\overline{\Lambda}_{a}(\theta) \tau_{1}}\E\prod_{j=1}^{\tau_{1}}\xi(X_{j-1},\Delta_j(X))\Big)}{\overline{E}^Q_0\Big(\tau_{1}\e^{\langle \theta,X_{\tau_{1}}\rangle -\overline{\Lambda}_{a}(\theta) \tau_{1}}\E\prod_{j=1}^{\tau_{1}}\xi(X_{j-1},\Delta_j(X))\Big)}
	\end{equation} and is positive definite. In particular, whenever $\mathrm{dis}(\P) < \gamma_1$ the set $\mathcal{A}_{y,\P}$ is open.
\end{prop}

\begin{proof} Taking derivatives on \eqref{eq:gradient_identity} (which again we can do by using Proposition \ref{prop:reglambdaa} and \eqref{eq:boundsup}) and proceeding as for \eqref{eq:rep} immediately yields \eqref{eq:hessian_formula}. On the other hand, for any column vector $v \in \R^{n\times 1}$ we have
	\[
	\langle v, H_a(\theta)\cdot v\rangle = \frac{\overline{E}^Q_0\Big(|\langle X_{\tau_{1}}-\nabla\overline{\Lambda}_{a}(\theta)\tau_{1}, v\rangle|^2\e^{\langle \theta,X_{\tau_{1}}\rangle -\overline{\Lambda}_{a}(\theta) \tau_{1}}\E\prod_{j=1}^{\tau_{1}}\xi(X_{j-1},\Delta_j(X))\Big)}{\overline{E}^Q_0\Big(\tau_{1}\e^{\langle \theta,X_{\tau_{1}}\rangle -\overline{\Lambda}_{a}(\theta) \tau_{1}}\E\prod_{j=1}^{\tau_{1}}\xi(X_{j-1},\Delta_j(X))\Big)},
	\] so that $\langle v, H_a(\theta)\cdot v\rangle \geq 0$ and the equality holds if and only if $\langle X_{\tau_{1}}-\nabla\overline{\Lambda}_{a}(\theta)\tau_{1}, v\rangle=0$ $\overline{Q}_0$-a.s. or, equivalently, if $\langle \tfrac{X_{\tau_1}}{\tau_1},v\rangle$ is $\overline{Q}_0$-almost surely constant. However, since $\inf_{e \in \bbV}\alpha(e)>0$, it is not hard to check that if $v \neq 0$ then $\langle \tfrac{X_{\tau_1}}{\tau_1},v\rangle$ cannot be constant. Hence, we see that in this case $v$ must be zero and therefore $H_a(\theta)$ is positive definite. Finally, that $\mathcal{A}_{y,\P}$ is open follows from this and the inverse function theorem.
\end{proof}

The next proposition states that, whenever the disorder is small enough, the set $\mathcal{A}_{y,\P}$ contains a ball centered at $\nabla \overline{\Lambda}_a(0)$ whose radius is independent of $\P$. 

\begin{prop}\label{prop2} There exist $\varepsilon_2=\varepsilon_2(y,d,\kappa),r_2=r_2(y,d,\kappa) > 0$ such that, for any $\P \in \mathcal{P}_\kappa$, if $\mathrm{dis}(\P)<\varepsilon_2$ then $B_{r_2}(\nabla\overline{\Lambda}_{a}(0))\subset \mathcal{A}_{y,\P}$.
\end{prop}

The proof of Proposition \ref{prop2} will be carried out in Subsection \ref{sec:proof:prop2} . As a consequence of Propositions \ref{lemma9} and \ref{prop2}, we immediately obtain the following corollary.

\begin{cor}\label{cor:ball}There exist $\varepsilon=\varepsilon(y,d,\kappa),r=r(y,d,\kappa) > 0$ such that, for any $\P \in \mathcal{P}_\kappa$, if $\mathrm{dis}(\P)<\varepsilon$ then $B_{r}(y)\subset \mathcal{A}_{y,\P}$.
\end{cor}

\subsection{Proof of Theorem \ref{th:main} and Theorem \ref{th:main2} (Assuming Proposition \ref{prop2}):}\label{sec:dev} Now, for $x\in B_{r}(y)$ with $r$ as in Corollary \ref{cor:ball}, define the quantities
\[
\tilde{I}_{q}(x):=\sup_{\theta \in \R^d}[\langle \theta,x\rangle-\overline{\Lambda}_{q}(\theta)]
\hspace{1cm}\text{ and }\hspace{1cm}
\tilde{I}_{a}(x):=\sup_{\theta \in \R^d}[\langle \theta,x\rangle-\overline{\Lambda}_{a}(\theta)].
\] It is standard to show that (see \cite[Lemma 2.3.9]{DZ98} for details) 
\begin{equation}\label{eq:ratef}
\tilde{I}_{q}(x)=\langle \theta_{x,q},y\rangle-\overline{\Lambda}_{q}(\theta_{x,q}) \hspace{1cm}\text{ and }\hspace{1cm}\tilde{I}_{a}(x)=\langle\theta_{x,a},y\rangle-\overline{\Lambda}_{a}(\theta_{x,a})
\end{equation}
for any $\theta_{x,q}$ and $\theta_{x,a}$ respectively satisfying 
\[
\nabla\overline{\Lambda}_{q}(\theta_{x,q})=x \hspace{1cm}\text{ and }\hspace{1cm}\nabla\overline{\Lambda}_{a}(\theta_{x,a})=x.
\] Notice that such $\theta_{x,a}$ exists and satisfies $|\theta_{x,a}|< \overline{\gamma}$ since $x \in \mathcal{A}_{y,\P}$ by choice of $x$. Furthermore, such $\theta_{x,q}$ also exists and in fact can be taken equal to $\theta_{x,a}$, since both $\overline{\Lambda}_q(\theta)$ and $\overline{\Lambda}_a(\theta)$ coincide for $|\theta| < \overline{\gamma}$ by Proposition \ref{prop1}. Hence, from \eqref{eq:ratef} and the fact that $\theta_{x,q}=\theta_{x,a}$, we obtain that $\tilde{I}_{q}(x)=\tilde{I}_{a}(x)$ for all $x \in B_{r}(y)$. We may then conclude Theorem \ref{th:main2} once we show this implies that $I_q(x)=I_a(x)$. But, from \eqref{eq:eqmgf} and the definition of $\tilde{I}_q$ and $\tilde{I}_a$, for $x \in B_{r}(y)$ we have that
\begin{equation}\label{eq:eqratefq}
\tilde{I}_{q}(x)+\log(\sqrt{C_{y,\alpha}}) + \langle \theta_{y,\alpha} ,x\rangle = \sup_{\theta \in \R^d}[\langle \theta+\theta_{y,\alpha},x\rangle - \Lambda_q(\theta+\theta_{y,\alpha})]=\sup_{\theta \in \R^d}[\langle \theta,x\rangle - \Lambda_q(\theta)] = I_q(x)
\end{equation} and
\begin{equation}\label{eq:eqratefa}
\tilde{I}_{a}(x)+\log(\sqrt{C_{y,\alpha}})+ \langle \theta_{y,\alpha} ,x\rangle  = \sup_{\theta \in \R^d}[\langle \theta+\theta_{y,\alpha},x\rangle - \Lambda_a(\theta+\theta_{y,\alpha})]=\sup_{\theta \in \R^d}[\langle \theta,x\rangle - \Lambda_a(\theta)] = I_a(x),
\end{equation}
where the rightmost equalities in \eqref{eq:eqratefq} and \eqref{eq:eqratefa} follow from standard arguments (see \cite[Section 2.3]{DZ98} for details) using that $\Lambda_q$ and $\Lambda_a$ are well-defined in the sense of Corollary \ref{cor:exist} and that $B_{r}(y)$ is contained in the set of exposed points of the Fenchel-Legendre transforms of both $\Lambda_q$ and $\Lambda_a$ by \eqref{eq:ratef} and \eqref{eq:eqmgf}. Therefore, as $\tilde{I}_{q}$ and $\tilde{I}_{a}$ agree on $B_{r}(y)$, we see that the same holds for $I_{q}, I_{a}$ and thus we obtain Theorem \ref{th:main2}. 

Then, in order to complete the proof, it only remains to prove Proposition \ref{prop2}. We do this next.

\subsection{Proof of Proposition \ref{prop2}}\label{sec:proof:prop2} The key ingredient in the proof of Proposition \ref{prop2} is the following uniform version of the inverse function theorem.

\begin{theorem}[Uniform inverse function theorem]\label{prop3}
	Let $\mathcal{F}$ be a family of $C^1$-functions $f:G \to \R^d$ defined on some neighborhood $G \subseteq \R^d$ of $0$ such that the differential matrix $Df(0) \in \R^{d\times d}$ is invertible for every $f \in \mathcal{F}$. Then, if there exist constants $c,\delta >0$ such that $\{ \theta : |\theta|<\delta\} \subseteq G$ and \begin{enumerate}
		\item [I1.] $\sup_{f \in \mathcal{F}} \norm{Df(0)^{-1}}<c$,
		\item [I2.] $\sup_{f \in \mathcal{F}\,,\,|\theta|<\delta}\norm{Df(\theta)-Df(0)}<\frac{1}{2c}$,
	\end{enumerate} 
	where $\norm{\cdot}$ denotes the operator $1$-norm, there exists $\rho$ (depending only on $c$ and $\delta$) such that for all $f \in \mathcal{F}$,
	\[
	B_\rho(f(0)) \subseteq \{ f(\theta) : |\theta| < \delta\}.
	\] 
\end{theorem}
The proof of Theorem \ref{prop3} is obtained by simply mimicking (part of) the proof of the standard inverse function theorem (see e.g. \cite[Theorem 9.24]{R76}), replacing the usual estimates with uniform bounds. Therefore, we omit the proof and leave the details to the reader. 

In light of Theorem \ref{prop3}, to obtain Proposition \ref{prop2} it will suffice to show that there exists $\varepsilon_2>0$ depending only on $y, d$ and $\kappa$ such that the family of $C^1$-functions 
\[
\mathcal{F}_y:=\{ \nabla \overline{\Lambda}_a : \P \in \mathcal{P}_\kappa \text{ with }\mathrm{dis}(\P)<\varepsilon_2\}
\] satisfies the hypotheses of Theorem \ref{prop3}. By Proposition \ref{prop:hessian}, we only need to check conditions (I1) and (I2). For this, we will need three auxiliary lemmas. The first one asserts that $\nabla\overline{\Lambda}_{a}(\theta)$ is close to $\nabla\overline{\Lambda}_{a}(0)$ (uniformly over $\P$) whenever $\theta$ is close to $0$ and the disorder is sufficiently small.

\begin{lemma}\label{lemma10}
	Given $c>0$, there exist $\eps_3=\varepsilon_3(y,c),\delta=\delta(y,c)>0$ such that, for any $\P \in \mathcal{P}_\kappa$, if~$\mathrm{dis}(\mathbb{P})<\eps_{3}$ then 
	\[
	\sup_{|\theta|<\delta}|\nabla\overline{\Lambda}_{a}(\theta)-\nabla\overline{\Lambda}_{a}(0)|<c.
	\]
\end{lemma}
\begin{proof} In view of \eqref{eq:bexp2} and the fact that $\overline{E}^Q_0(\tau_1) \geq 1$, it will be enough to check that, given $c'>0$, there exist $\varepsilon'=\varepsilon'(y,c'),\delta'=\delta'(y,c')>0$ such that, for any $\mathbb{P} \in \mathcal{P}_\kappa$, if $\mathrm{dis}(\mathbb{P})< \varepsilon'$ then 
	\[
	\sup_{|\theta|<\delta'}\Big| \overline{E}^Q_0\Big(X_{\tau_{1}}\E\prod_{j=1}^{\tau_{1}}\xi(X_{j-1},\Delta_j(X))\Big(\e^{\langle \theta,X_{\tau_{1}}\rangle -\overline{\Lambda}_{a}(\theta) \tau_{1}}-\e^{-\overline{\Lambda}_{a}(0) \tau_{1}}\Big)\Big)\Big|<c'
	\] and
	\[
	\sup_{|\theta|<\delta'}\Big| \overline{E}^Q_0\Big(\tau_{1}\E\prod_{j=1}^{\tau_{1}}\xi(X_{j-1},\Delta_j(X))\Big(\e^{\langle \theta,X_{\tau_{1}}\rangle -\overline{\Lambda}_{a}(\theta) \tau_{1}}-\e^{-\overline{\Lambda}_{a}(0) \tau_{1}}\Big)\Big)\Big|<c'.
	\] But this can be done exactly as in the proof of \eqref{eq:bexp1}-\eqref{eq:bexp2}, using now the inequality
	\[
	|\langle \theta,X_{\tau_{1}}\rangle -\overline{\Lambda}_{a}(\theta)\tau_1| + |\overline{\Lambda}_{a}(0)\tau_1| \leq 2\big(|\theta|+h(\mathrm{dis}(\P))\big)\tau_1,
	\] where $h$ is as in Remark \ref{rem:bound}, which follows in the same way as the inequalities in this last remark. We omit the details.
\end{proof}

The second lemma is the analogue of Proposition \ref{lemma9} but for the Hessian $H_a$, which states that 
whenever $\mathrm{dis}(\P)$ is sufficiently small $H_a(0)$ will be close to the corresponding Hessian for the case of zero disorder. 

\begin{lemma}\label{lemma:h2} Given $c > 0$, there exist $\varepsilon_4=\varepsilon_4(y,c)>0$ such that, for any $\P \in \mathcal{P}_\kappa$, if $\mathrm{dis}(\mathbb{P})<\varepsilon_4$ then
	\[
	\big\lVert H_a(0)-H^*_a(0)\big\rVert < c,
	\] where
	\begin{equation}\label{eq:defhstar}
	H^*_a(0):=\frac{\overline{E}^Q_0((X_{\tau_1}-y\tau_1)^T(X_{\tau_1}-y\tau_1))}{\overline{E}^Q_0(\tau_1)}.
	\end{equation}
\end{lemma}
\begin{proof} For simplicity, let us set $\Gamma(v):=(X_{\tau_1}-v\tau_1)^T(X_{\tau_1}-v\tau_1)$ for $v \in \R^d$. Then, in view of \eqref{eq:bexp2}, the fact that $\overline{E}^Q_0(\tau_1) \geq 1$ and since
	\[
	\lVert \overline{E}^Q_0(\Gamma(y)) \rVert \leq \overline{E}^Q_0 (|X_{\tau_1}-y\tau_1|^2) \leq (1+|y|)^2\overline{E}^Q_0(\tau_1^2),
	\] by Proposition \ref{prop:exp_moments} (which can be used to bound the second moment of $\tau_1$ uniformly in $\P$) we see that it will suffice to show that the numerators of both matrices are close, i.e. that given any $c'>0$, there exists $\varepsilon'=\varepsilon'(y,c')>0$ such that if $\mathrm{dis}(\P) < \varepsilon'$ then
	\begin{equation}\label{eq:bhess1}
	\bigg\lVert \overline{E}^Q_0\Big(\Gamma(\nabla \overline{\Lambda}_a(0))\e^{-\overline{\Lambda}_{a}(0) \tau_{1}}\E\prod_{j=1}^{\tau_{1}}\xi(X_{j-1},\Delta_j(X))\Big) - \overline{E}^Q_0(\Gamma(y))\bigg\rVert < c'.
	\end{equation} Now, writing $\Xi_a(0):=\e^{-\overline{\Lambda}_{a}(0) \tau_{1}}\E\prod_{j=1}^{\tau_{1}}\xi(X_{j-1},\Delta_j(X))$ for simplicity, observe that we can bound the left-hand side of \eqref{eq:bhess1} from above by
	\[
	\overline{E}^Q_0\Big(\lVert\Gamma(\nabla \overline{\Lambda}_a(0))-\Gamma(y)\rVert\, |\Xi_a(0)|\Big) + \overline{E}^Q_0\Big(\lVert\Gamma(y) \rVert \,|\Xi_a(0)-1|\Big).
	\]
	Since by Remark \ref{rem:bound} we have
	\begin{equation}\label{eq:bhess2}
	|\Xi_a(0) -1 | \leq h(\mathrm{dis}(\P))\tau_1\e^{h(\mathrm{dis}(\P))\tau_1}, 
	\end{equation} and, furthermore, it is straightforward to verify that 
	\begin{equation}\label{eq:bhess3}
	\lVert\Gamma(\nabla \overline{\Lambda}_a(0))-\Gamma(y)\rVert \leq 5(|\nabla \overline{\Lambda}_a(0)-y| \vee 1)\tau_1^2
	\end{equation}
	and 
	\begin{equation}\label{eq:bhess4}
	\lVert \Gamma(y) \rVert \leq |X_{\tau_1}-y\tau_1|^2 \leq (1+|y|)^2\tau_1^2,		
	\end{equation}\eqref{eq:bhess1} follows at once from \eqref{eq:bhess2}-\eqref{eq:bhess3}-\eqref{eq:bhess4} by using Propositions \ref{prop:exp_moments} and \ref{lemma9}. 
\end{proof}

The last auxiliary lemma states that $\lVert (H^*_a(0))^{-1}\rVert$ is uniformly bounded over $\mathcal{P}_\kappa$.

\begin{lemma}\label{lemma:unifstar} The mapping $\alpha \mapsto H_a^*(0)$ is continuous on $\mathcal{M}_1^*(\bbV):=\{ \alpha \in \mathcal{M}_1(\bbV) : \inf_{e \in \bbV} \alpha(e)> 0 \}$. In particular, for any $\kappa > 0$ we have $\sup_{\P \in \mathcal{P}_k} \lVert (H^*_a(0))^{-1}\rVert < \infty$.
\end{lemma}

\begin{proof} By definition of $H^*_a(0)$, it suffices to check that the mappings 
	\[
	\alpha \mapsto \overline{E}^Q_0(\tau_1) \quad \text{ and }\quad \alpha \mapsto \overline{E}^Q_0((X_{\tau_1}-y\tau_1)^T(X_{\tau_1}-y\tau_1))
	\] are continuous on $\mathcal{M}_1^*(\bbV)$. The proof for both mappings is similar, so we only show the continuity of $\alpha \mapsto \overline{E}^Q_0(\tau_1)$. To this end, since $\overline{E}^Q_0(\tau_1\1_{\{\tau_1 > N\}}) \to 0$ as $N \to \infty$ uniformly over $\mathcal{M}_1^*(\bbV)$ by Proposition \ref{prop:exp_moments}, it will be enough to show that $\alpha \mapsto \overline{E}^Q_0(\tau_1{\1}_{\{\tau_1 = N\}})$ is continuous for every $N \geq 1$. But, using the Markov property together with the fact that $Q_x(\beta_0=\infty)$ does not depend on $x$, it is not difficult to see that $\overline{E}^Q_0(\tau_1{\1}_{\{\tau_1 = N\}})$ is a polynomial of degree $N$ in the weights $u=(u(e))_{e \in \bbV}$ from \eqref{eq:w1}. Indeed, we have
	\[
	\overline{E}^Q_0(\tau_1{\1}{1}_{\{\tau_1 = N\}}) = \sum_{\bar{x}_n} \prod_{j=1}^n \alpha(\Delta_j(\bar{x}_n)),
	\] where the sum is over all paths $\bar{x}_n$ of length $n$ which start at $0$ and be extended to an infinite path $\bar{x}_\infty$ such that $\tau_1(\bar{x}_\infty)=n$, where $\tau_1(\bar{x}_\infty)$ denotes the analogue of $\tau_1$ but for $\bar{x}_\infty$. Therefore, since the weights $u(e)$ all depend continuously on $\alpha$, the continuity of $\alpha \mapsto \overline{E}^Q_0(\tau_1\1{1}_{\{\tau_1 = N\}})$ follows.
	
	Finally, to check the last statement, we first notice that $\alpha \mapsto \lVert (H^*_a(0))^{-1}\rVert$ is also continuous on $\mathcal{M}_1^*(\bbV)$ by Proposition \ref{prop:hessian}, since the mappings $A \mapsto A^{-1}$ and $A \mapsto \lVert A \rVert$ are also continuous in their respective domains. Hence, since $\mathcal{M}^{(\kappa)}_1(\bbV)$ is compact for any $\kappa > 0$ and 
	\[
	\sup_{\P \in \mathcal{P}_k} \lVert (H^*_a(0))^{-1}\rVert = \sup_{\alpha \in \mathcal{M}^{(\kappa)}_1(\bbV)} \lVert (H^*_a(0))^{-1}\rVert,
	\] the last statement now follows.
\end{proof}

We are now ready to show (I1) and (I2). To check (I1), using Lemmas \ref{lemma:h2}-\ref{lemma:unifstar} we may choose $\varepsilon_2 > 0$ depending only on $y, d$ and $\kappa$ such that if $\mathrm{dis}(\P)< \varepsilon_2$ then 
\[\lVert H_a(0) - H^*_a(0) \rVert \leq \frac{1}{2 \sup_{\P \in \mathcal{P}_k} \lVert (H^*_a(0))^{-1}\rVert}.
\] Then, using the identity $A^{-1}-B^{-1}=A^{-1}(B-A)B^{-1}$ for any invertible matrices $A,B \in \R^{d\times d}$, we have that, for any $\P \in \mathcal{P}_k$, if $\mathrm{dis}(\P)< \varepsilon_2$ then 
\[
\norm{(H_{a}(0))^{-1}-(H^*_a(0))^{-1}}\leq \norm{(H_{a}(0))^{-1}}\norm{H_{a}(0)-H^*_a(0)}\norm{(H^*_a(0))^{-1}}< \frac{1}{2}\norm{(H_{a}(0))^{-1}},
\] so that by the triangle inequality
\[
\norm{H_{a}(0)^{-1}}\leq \frac{1}{2}\norm{H_{a}(0)^{-1}}+ \norm{(H^*_a(0))^{-1}}
\] and thus 
\[
\norm{H_{a}(0)^{-1}} \leq 2 \norm{(H^*_a(0))^{-1}} \leq 2\sup_{\P \in \mathcal{P}_k} \lVert (H^*_a(0))^{-1}\rVert.
\] This shows (I1) for $c:= 2\sup_{\P \in \mathcal{P}_k} \lVert (H^*_a(0))^{-1}\rVert$. It remains to check (I2).

By arguing as in the proof of Lemma \ref{lemma:h2}, to check (I2) it will suffice to show that, given $c'>0$, one can find $\varepsilon'_2=\varepsilon'_2(y,c'),\delta=\delta(y,c')>0$ such that if $\mathrm{dis}(\P)< \varepsilon'_2$ then 
\[
\sup_{|\theta|< \delta} \Big\lVert \overline{E}^Q_0\big(\Gamma(\nabla \overline{\Lambda}_a(\theta))\Xi_a(\theta)\big) - \overline{E}^Q_0\big(\Gamma(\nabla \overline{\Lambda}_a(0))\Xi_a(0)\big)\Big\rVert < c'
\] where, for $v,\theta \in \R^d$, we set
\[
\Gamma(v):=(X_{\tau_1}-v\tau_1)^T(X_{\tau_1}-v\tau_1) \hspace{1cm}\text{ and }\hspace{1cm}\Xi_a(\theta):=\e^{\langle \theta,X_{\tau_{1}}\rangle -\overline{\Lambda}_{a}(\theta) \tau_{1}}\E\prod_{j=1}^{\tau_{1}}\xi(X_{j-1},\Delta_j(X)).
\] But this can be done as in the proof of Lemma \ref{lemma:h2}, by using Lemma \ref{lemma10} and \eqref{eq:bhess2}-\eqref{eq:bhess3}-\eqref{eq:bhess4} together with the inequalities 
\[
\lVert \Gamma(\nabla \overline{\Lambda}_a(\theta)) - \Gamma(\nabla \overline{\Lambda}_a(0))\rVert \leq 5(|\nabla \overline{\Lambda}_a(\theta)-\nabla \overline{\Lambda}_a(0)| \vee 1)\tau_1^2
\] and 
\[
|\Xi_a(\theta)-\Xi_a(0)|\leq 2(|\theta|+h(\mathrm{dis}(\P)))\tau_1\e^{2(|\theta|+h(\mathrm{dis}(\P)))\tau_1}
\] for $h$ as in Remark \ref{rem:bound}, which are both straightforward to check. This shows (I2) and therefore completes the proof of Proposition \ref{prop2}.

\section{Non-triviality of $\lim_{n \to \infty} \Phi_n(\theta)$ - proof of Propositions \ref{prop:norm1} and \ref{prop:norm2}}\label{sec:l12}

\subsection{Proof of Proposition \ref{prop:norm1}} The first step in the proof will be to show that there exists $\gamma_1=\gamma_1(y)>0$ such that, for any $\P \in \mathcal{P}_\kappa$, whenever $|\theta| \vee \mathrm{dis}(\mathbb{P}) < \gamma_{1}$ we have that \eqref{eq:m1} holds. This will be a consequence of the following two lemmas.

\begin{lemma}\label{lemma1}
	For all $\theta\in \R^{d}$,
	\[
	\overline{E}^Q_0\Big(\e^{\langle \theta,X_{\tau_{1}}\rangle -\overline{\Lambda}_{a}(\theta) \tau_{1}}\E\prod_{j=1}^{\tau_{1}}\xi(X_{j-1},\Delta_j(X))\Big)\leq 1.
	\]
\end{lemma}

\begin{lemma}\label{lemma3} There exists $\gamma_1=\gamma_1(y)>0$ such that, for any $\P \in \mathcal{P}_\kappa$, whenever $|\theta| \vee \mathrm{dis}(\mathbb{P}) < \gamma_{1}$,
	\begin{equation}\label{eq5}
	\overline{E}^Q_0\Big(\e^{\langle \theta,X_{\tau_{1}}\rangle -\overline{\Lambda}_{a}(\theta) \tau_{1}}\E\prod_{j=1}^{\tau_{1}}\xi(X_{j-1},\Delta_j(X))\Big)\geq 1.
	\end{equation}
\end{lemma}

Postponing the proofs of these lemmas for a moment, let us finish the proof of Proposition \ref{prop:norm1}. For $\theta \in \R^d, \P \in \mathcal{P}_\kappa$ such that $|\theta| \vee \mathrm{dis}(\P) < \gamma_1$ we may define the probability measure $\mu^{(\theta)}$ on $\Z^d$ as
\begin{equation}\label{eq:Probtheta}
\mu^{(\theta)}(x):=\overline{E}^Q_0\Big(\e^{\langle \theta,X_{\tau_{1}}\rangle -\overline{\Lambda}_{a}(\theta) \tau_{1}}\E\prod_{j=1}^{\tau_{1}}\xi(X_{j-1},\Delta_j(X))\,;\, X_{\tau_{1}}=x\Big)
\end{equation} and consider the random walk $Y^{(\theta)}=(Y^{(\theta)}_n)_{n \in \N_0}$ with jump distribution $\mu^{(\theta)}$. Then, if $\widehat{E}^{(\theta)}_0$ denotes expectation with respect to $\widehat{P}^{(\theta)}$, the law of $Y^{(\theta)}$ starting from $0$, we have that
\begin{equation}\label{eq:rwt}
\lim_{n \to \infty} \E \Phi_n(\theta) = \frac{1}{\widehat{E}^{(\theta)}_0( \langle Y_n , \ell \rangle )}.
\end{equation} Indeed, using \eqref{eq:m1} and the renewal structure of the $Q$-random walk, for each $n \geq 1$ we have
\begin{align}
\E \Phi_{n}(\theta) &=\sum_{k=1}^{\infty}\overline{E}^Q_0\Big(\e^{\langle \theta,X_{\tau_{k}}\rangle -\overline{\Lambda}_{a}(\theta) \tau_{k}}\E\prod_{j=1}^{\tau_{k}}\xi(X_{j-1},\Delta_j(X))\,;\,L_{n}=\tau_{k}\Big) \nonumber \\
&=\sum_{k=1}^{\infty}\widehat{P}^{(\theta)}_0(\langle Y_{k},\ell\rangle=n)= \widehat{P}^{(\theta)}_0( \langle Y_k,\ell \rangle = n \text{ for some }k \geq 1) \label{eq:prob1}
\end{align} so that \eqref{eq:rwt} is now a consequence of the renewal theorem for the sequence $(\langle Y_k - Y_{k-1},\ell\rangle)_{k \geq 1}$. Finally, Proposition \ref{prop:norm1} then follows \eqref{eq:rwt} and the next lemma. 

\begin{lemma}\label{lema:rwt} There exists $\gamma_1=\gamma_1(y)>0$ such that, for any $\P \in \mathcal{P}_\kappa$, whenever $|\theta| \vee \mathrm{dis}(\P) < \gamma_1$, 
	\[
	\widehat{E}^{(\theta)}_0( \langle Y_n , \ell \rangle )< \infty.
	\]
\end{lemma} Thus, in order to complete the proof of Proposition \ref{prop:norm1} we only need to prove Lemmas \ref{lemma1}, \ref{lemma3} and \ref{lema:rwt} above. The rest of this subsection is devoted to this. 

\begin{proof}[Proof of Lemma \ref{lemma1}] Given $\delta>0$, let us write $\eta_{\theta,\delta}:=\overline{\Lambda}_{a}(\theta)+\delta$ for simplicity and for $n \geq 1$ define 
	\[
	\Upsilon_{n,\delta}(\theta):= E^Q_0\Big(\e^{\langle \theta,X_{\tau_{n}}\rangle -\eta_{\theta,\delta} \tau_{n}}\E\prod_{j=1}^{\tau_{n}}\xi(X_{j-1},\Delta_j(X))\Big).
	\]Then, by splitting the expectation in the definition of $\Upsilon_{n,\delta}(\theta)$ according to the different possible values for $\tau_n$, we have as in \eqref{eq:bound1} that 
	\begin{equation}\label{eq:bound2}
	\Upsilon_{n,\delta}(\theta) \leq \sum_{k=n}^{\infty}\e^{-\eta_{\theta,\delta}k}E^Q_0\Big(\e^{\langle \theta,X_{k}\rangle}\E\prod_{j=1}^{k}\xi(X_{j-1},\Delta_j(X))\Big)\\
	\end{equation} Since, for some $o(1) \to 0$ as $k \to \infty$ we have
	\begin{equation}\label{eq:convannealed}
	E^Q_0\Big(\e^{\langle \theta, X_{k}\rangle}\E\prod_{j=1}^{k}\xi(X_{j-1},\Delta_j(X))\Big) = \e^{(\overline{\Lambda}_{a}(\theta)+o(1))k},
	\end{equation} from \eqref{eq:bound2} we obtain that for all $n$ sufficiently large (depending on $\delta$)
	\begin{equation}\label{eq:boundupsilon}
	\Upsilon_{n,\delta}(\theta) \leq \sum_{k=n}^\infty \e^{-\tfrac{\delta}{2}k} = \frac{\e^{-\tfrac{\delta}{2}n}}{1-\e^{-\tfrac{\delta}{2}}}.
	\end{equation} On the other hand, by the renewal structure, we have $Q_0$-almost surely, 
	\begin{equation}\label{eq:descomptau}
	\E\prod_{j=1}^{\tau_{n}}\xi(X_{j-1},\Delta_j(X)) = \prod_{i=0}^{n-1} \left(\E \prod_{j=\tau_i+1}^{\tau_{i+1}}\xi(X_{j-1},\Delta_j(X))\right).
	\end{equation} From this, using the renewal structure once again together with the translation invariance of $\P$, we see that for all $n \geq 1$
	\begin{equation}\label{eq:renew}
	\Upsilon_{n,\delta}(\theta) = \Upsilon_{1,\delta}(\theta) \left( \overline{E}^Q_0\Big(\e^{\langle \theta,X_{\tau_{1}}\rangle -\eta_{\theta,\delta}\tau_{1}}\E\prod_{j=1}^{\tau_{1}}\xi(X_{j-1},\Delta_j(X))\Big)\right)^{n-1}.
	\end{equation} Since $\Upsilon_{1,\delta}(\theta)>0$, in light of \eqref{eq:boundupsilon} we conclude that 
	\[
	\overline{E}^Q_0\Big(\e^{\langle \theta,X_{\tau_{1}}\rangle -\overline{\Lambda}_{a}(\theta) \tau_{1}}\E\prod_{j=1}^{\tau_{1}}\xi(X_{j-1},\Delta_j(X))\Big) \leq \e^{-\tfrac{\delta}{2}}.
	\] Letting $\delta \searrow 0$, by monotone convergence we get the desired result.
\end{proof}

\begin{proof}[Proof of Lemma \ref{lemma3}] Given $\theta \in \R^d$, $n\geq 1$ and $r\in \R$, let us write
	\begin{equation}\label{eq:defxi}
	\Xi_{n,r}(\theta):=\e^{\langle \theta, X_{n}\rangle-rn}\E\prod_{j=1}^{n}\xi(X_{j-1},\Delta_j(X)).
	\end{equation} Then, by splitting $E^Q_0(\Xi_{n,r}(\theta))$ according to the different events $\{n \in (\tau_m,\tau_{m+1}]\,,\,n=\tau_m+i\}$ for $m=0,\dots,n-1$ and $i=1,\dots,n$ and using the Markov property at $\tau_m$, we see that
	\begin{align}
	E^Q_0(\Xi_{n,r}(\theta)) &\leq \sum_{m=0}^{n-1}\sum_{i=1}^n E^Q_0(\Xi_{\tau_m,r}(\theta)\,;\,\tau_m=n-i) \overline{E}^Q_0(\Xi_{i,r}(\theta)\,;\,\tau_1> i) \nonumber\\ \nonumber &\leq \sum_{m=0}^{n-1} E^Q_0(\Xi_{\tau_m,r}(\theta)) \overline{E}^Q_0\Big(\sup_{i \leq \tau_1} \Xi_{i,r}(\theta)\Big) \nonumber \\
	& \leq \overline{E}^Q_0\Big(\sup_{i \leq \tau_1} \Xi_{i,r}(\theta)\Big) \left( 1 + E^Q_0\Big(\sup_{i \leq \tau_1} \Xi_{i,r}(\theta)\Big)\sum_{m=1}^{\infty} \Big(\overline{E}^Q_0(\Xi_{\tau_1,r})\Big)^{m-1}\right) \label{eq:boundunif}
	\end{align} where, in order to obtain the last inequality, we have used that for $m \geq 1$,
	\[
	E^Q_0(\Xi_{\tau_m,r}(\theta)) =E^Q_0(\Xi_{\tau_1,r}) \Big(\overline{E}^Q_0(\Xi_{\tau_1,r})\Big)^{m-1}
	\] which follows from the renewal structure as in \eqref{eq:renew}.
	
	Now, if we take then $r=\overline{\Lambda}_a(\theta) -\delta$ for some $\delta> 0$ then by Remark \ref{rem:bound} we have, for any $i \geq 1$,
	\[
	\Xi_{i,r}(\theta) \leq \exp\left( \big(2|\theta|+h(\mathrm{dis}(\P))+\delta\big)i\right).
	\] If we choose $\gamma_1$ and $\delta$ small enough (but depending only on $y$) so that $2|\theta|+h(\mathrm{dis}(\P)) +\delta < \tfrac{\gamma_0}{2}$ whenever $|\theta|\vee \mathrm{dis}(\P) < \gamma_1$, where $\gamma_{0}$ is as in Proposition \ref{prop:exp_moments}, then we obtain that
	\begin{equation}\label{eq:boundsup}
	E^Q_0\Big(\sup_{i \leq \tau_1} \Xi_{i,r}(\theta)\Big) \leq E^Q_0\Big(\e^{\tfrac{\gamma_{0}}{2}\tau_1}\Big)< \infty,
	\end{equation} and combining \eqref{eq:boundsup}  with Lemma \ref{lemma:pos} shows that $\overline{E}^Q_0(\sup_{i \leq \tau_1} \Xi_{i,r}(\theta))<\infty$ as well. Thus, since the bound in \eqref{eq:boundunif} is uniform in $n$, if $\overline{E}^Q_0\big(\Xi_{\tau_1,r}(\theta)\big)< 1$ then we would have $\sup_{n \geq 1} E^Q_0( \Xi_{n,r}(\theta))< \infty$, and this in turn would imply that
	\[
	\lim_{n \to \infty} \frac{1}{n} \log E^Q_0( \Xi_{n,r}(\theta)) = 0.
	\] However, observe that by choice of $r$, definition of $_{n,r}$ and \eqref{eq:convannealed}, we have that
	\[
	\lim_{n \to \infty} \frac{1}{n} \log E^Q_0( \Xi_{n,r}(\theta)) = \delta
	\] so that in reality whenever $|\theta|\vee \mathrm{dis}(\P) < \gamma_1$ we must have
	\[
	1 \leq \overline{E}^Q_0\big(\Xi_{_1,r}(\theta)\big)=\overline{E}^Q_0\Big(\e^{\langle \theta, X_{\tau_1}\rangle-(\overline{\Lambda}_a(\theta)-\delta)\tau_1}\E\prod_{j=1}^{\tau_1}\xi(X_{j-1},\Delta_j(X))\Big). 
	\] Letting $\delta \searrow 0$, by dominated convergence we get the desired result (note that we can  indeed use dominated convergence since $\overline{E}^Q_0\big(\Xi_{\tau_1,r}(\theta)\big)<\infty$ for  $r=\overline{\Lambda}_a(\theta) -\delta$ and $\delta>0$ sufficiently small, by \eqref{eq:boundsup} and choice of $\gamma_0$). This concludes the proof.
\end{proof}

\begin{proof}[Proof of Lemma \ref{lema:rwt}] Since $\langle Y_1 , \ell \rangle \leq \tau_1$ by definition of $\tau_1$, using also that $\tau_1 \leq \frac{1}{\delta}\e^{\delta \tau_1}$ for any $\delta >0$, we see that
	\[
	\widehat{E}^{(\theta)}_0( \langle Y_n , \ell \rangle ) \leq \frac{1}{\delta} \overline{E}^Q_0\Big(\e^{\langle \theta,X_{\tau_{1}}\rangle -(\overline{\Lambda}_{a}(\theta)-\delta) \tau_{1}}\E\prod_{j=1}^{\tau_{1}}\xi(X_{j-1},\Delta_j(X))\Big)
	\] and so the lemma now follows as in the proof of \eqref{eq:boundsup}.	
\end{proof}

\subsection{Proof of Proposition \ref{prop:norm2}} We will show that there exists a constant $\gamma_2 > 0$, depending only on $y,d$ and $\kappa$ such that, for any $\P \in \mathcal{P}_\kappa$, if $\mathrm{dis}(\mathbb{P}) < \gamma_2$ then 
\[
\sup_{n\geq 1\,,\,|\theta|<\gamma_2}\E (\Phi_{n}(\theta))^{2}<\infty.
\]
This is equivalent to showing that
\begin{equation}\label{eq:to_prove_2}
\sup_{n\geq 1\,,\,|\theta| < \gamma_2} \overline{E}^Q_{0,0}\Big(\e^{\langle \theta,  X_{L_{n}}+\widetilde{X}_{\widetilde{L}_{n}}\rangle-\overline{\Lambda}_{a}(\theta)(L_{n}+\widetilde{L}_n)} \E\prod_{j=1}^{L_{n}}\xi(X_{j-1},\Delta_j(X))\prod_{j=1}^{\widetilde{L}_{n}}\xi(\widetilde{X}_{j-1},\Delta_j(\widetilde{X}))\,;\, n\in \mathcal{L}\Big) < \infty,
\end{equation}
where $X=(X_n)_{n\in \N_0}$ and $\widetilde{X}=(\widetilde{X}_n)_{n \in \N_0}$ are independent copies of the conditioned random walk with law $\overline{Q}_0$, $\widetilde{L}_n$ and $\widetilde{\tau}_{n}$ are the analogues of $L_n$ $\tau_{n}$ but for $\widetilde{X}$, and 
\begin{equation}\label{eq: reg_levels}
\mathcal{L}:=\{n\geq 0: \langle X_{i},\ell\rangle \geq n \text{ for all } i\geq L_{n}\,,\,\langle \widetilde{X}_{j},\ell\rangle \geq n\text{ for all } j\geq \widetilde{L}_{n}\}
\end{equation} are the so-called \textit{common renewal levels}. In the sequel, we shall write $\overline{Q}_{x,\widetilde{x}}:=\overline{Q}_x \times \overline{Q}_{\widetilde{x}}$ and $\overline{E}^Q_{x,\widetilde{x}}$ to denote expectation with respect to $\overline{Q}_{x,\widetilde{x}}$. 

In order to check \eqref{eq:to_prove_2}, let us introduce, for $x \in \Z^d$, $e \in \bbV$ and $n\geq 1$, the quantities
\[
N_{x,e}(n):=\# \{ j \in \{1,\dots, n\} : X_{j-1}=x\,,\,\Delta_{j}(X)=e\}=\sum_{j=1}^n \1{1}_{x}(X_{j-1})\1{1}_e(\Delta_{j}(X))
\] and 
\[
N_{x}(n):=\#\{ j \in \{1,\dots, n\} : X_{j-1}=x\}=\sum_{e \in \bbV} N_{x,e}(n),
\] as well as the corresponding analogues $\widetilde{N}_{x,e}(n)$ and $\widetilde{N}_x(n)$ for $\widetilde{X}$. Then, using that by definition of $\mathrm{dis}(\P)$ we have that, for all $x \in \Z^d$, $e \in \bbV$ and $h$ as in Remark \ref{rem:bound}, the inequality
\[
\omega(x,e) \leq \widetilde{\omega}(x,e)\e^{h(\mathrm{dis}(\P))}
\] holds almost surely for any pair of independent environments $\omega$ and $\widetilde{\omega}$ with law $\P$, we have
\begin{align*}
\E \prod_{j=1}^{L_{n}}&\omega(X_{j-1},\Delta_j(X))\prod_{j=1}^{\widetilde{L}_{n}}\omega(\widetilde{X}_{j-1},\Delta_j(\widetilde{X})) =\prod_{x \in \Z^d}\E\prod_{e \in \bbV}\omega(x,e)^{N_{x,e}(L_{n})+\widetilde{N}_{x,e}(\widetilde{L}_{n})}\\
&\leq \prod_{x \in \Z^d}\E\left[\prod_{e \in \bbV}\omega(x,e)^{N_{x,e}(L_{n})}\right]\E\left[\prod_{e \in \bbV}\omega(x,e)^{\widetilde{N}_{x,e}(\widetilde{L}_{n})}\right]\e^{h(\mathrm{dis}(\P)) [N_{x}(L_{n})\wedge \widetilde{N}_{x}(\widetilde{L}_{n})]}\\
&=\E\left[\prod_{j=1}^{L_{n}}\omega(X_{j-1},\Delta_j(X)) \right]\E\left[\prod_{j=1}^{\widetilde{L}_{n}}\omega(\widetilde{X}_{j-1},\Delta_j(\widetilde{X})) \right]\e^{h(\mathrm{dis}(\P)) I_n},
\end{align*}
where 
$$
I_n:=\sum_{x \in \Z^d}[N_{x}(L_{n})\wedge \widetilde{N}_{x}(\widetilde{L}_{n})].
$$
 Hence, we conclude that the supremum in \eqref{eq:to_prove_2} is bounded from above by 
\begin{equation}\label{eq:to_prove_3}
A:=\sup_{n \geq 1\,,\,|\theta|<\gamma_2\,,\, z\in \bbV_{d}}A_{z,n}(\theta),
\end{equation}
where, for $z \in \bbV_{d}:=\{z\in \Z^{d}:\langle z, \ell\rangle=0\}$ and $n \geq 1$, we define
\begin{equation}\label{eq:An}
A_{z,n}(\theta):=\overline{E}^Q_{0,z}\big(F_n(\theta)\,;\,n\in \mathcal{L}\big)
\end{equation} with 
$$
F_n(\theta):= \phi_n(\theta)\widetilde{\phi}_n(\theta)\e^{h(\mathrm{dis}(\P)) I_n},
$$
 where
\begin{equation}\label{eq:Gn}
\phi_n(\theta):=\e^{\langle \theta, X_{L_n}-X_0\rangle-\overline{\Lambda}_{a}(\theta)L_{n}} \E\prod_{j=1}^{L_{n}}\xi(X_{j-1}-X_0,\Delta_j(X))
\end{equation} and $\widetilde{\phi}_n(\theta)$ is defined analogously but interchanging $(X,L_n)$ with $(\widetilde{X},\widetilde{L}_n)$. 

In order to prove Proposition \ref{prop:norm2}, we will show that $A$ is finite provided that $\theta \vee \mathrm{dis}(\P)$ is taken sufficiently small (depending only on $y,d$ and $\kappa$). To this end, let us set
\begin{equation}
\zeta:=\inf\{m\geq 0: \exists\, i,j \geq 1 \text{ such that } X_{i}=\widetilde{X}_{j}\text{ and }\langle X_{i},\ell\rangle=m\},
\end{equation}
i.e. the first level in which both walks intersect at a time other than zero. Observe that whenever $1 \leq n\leq \zeta$ we have $X_i \neq \widetilde{X}_j$ for all $i < L_n$ and $j < \widetilde{L}_n$, so that $I_n = 1 \leq 1$, with the only possible non-vanishing term being $x=0$. In particular, by virtue of independence and the definition of $\mathcal{L}$, we obtain that, for $\gamma_1=\gamma_1(y)>0$ as in the proof of Proposition \ref{prop:norm1} and any $\P \in \mathcal{P}_\kappa$, whenever $|\theta| \vee \mathrm{dis}(\P) < \gamma_1 \wedge \frac{1}{2}$ we have
\begin{align*}
\overline{E}^Q_{0,z}\big( F_n(\theta)\,;\,n\in \mathcal{L}\,,\,n \leq \zeta\big) & \leq \overline{E}^Q_{0,z}\Big( \phi_n(\theta)\widetilde{\phi}_n(\theta)\e^{h(1/2)},n\in \mathcal{L}\Big)\\
&=\e^{h(1/2)} \Big[\E \Phi_n(\theta) \Big]^2 \leq \e^{h(1/2)}
\end{align*} where for the last inequality we have used that $\E \Phi_n(\theta) \leq 1$ since it coincides with a probability by \eqref{eq:prob1}. In light of this bound we see that, in order to show that $A$ is finite, it only remains to obtain a suitable control on the expectation
\begin{equation}\label{eq:secondterm}
\overline{E}^Q_{0,z}\Big( F_n(\theta)\,;\,n\in \mathcal{L}\,,\,n > \zeta\Big).
\end{equation} 

To this end, define
\begin{equation}\label{eq:sigma1}
\sigma:=\inf\{k\in \mathcal{L}:k>\zeta\},
\end{equation}
i.e. the first common renewal level after the walks first intersect (at a time other than zero). Then, by \eqref{eq:descomptau}, the Markov property and translation invariance, \eqref{eq:secondterm} can be rewritten as
\begin{align*}
\sum_{k=1}^{n}\overline{E}^Q_{0,z}\Big( F_n(\theta)&\,;\,n\in \mathcal{L}\,,\,\sigma=k\Big)\\
&=\sum_{k=1}^{n}\sum_{z'\in \bbV_{d}}\overline{E}^Q_{0,z}\Big( F_k(\theta)\, ;\,\sigma=k\,,\,\widetilde{X}_{\widetilde{L}_{k}}-X_{L_{k}}=z'\Big)
\overline{E}^Q_{0,z'}\Big( F_{n-k}(\theta)\,;\,n-k\in \mathcal{L}\Big)\\
&\leq \sum_{k=1}^{n}\overline{E}^Q_{0,z}\Big( F_k(\theta)\,;\,\sigma=k\Big)\sup_{z'\in \bbV_{d}}A_{z',n-k}(\theta),
\end{align*}
where we use the convention $A_{z',0}(\theta):= 1$ and, to obtain the first equality, we have used that $N_x(L_k) = N_x(L_n)$ whenever $\langle x,\ell \rangle < k$ and $N_x(L_k) = 0$ whenever $\langle x,\ell\rangle \geq k$ (and the analogous statements for $\widetilde{N}_x$). Now, if we set 
\begin{equation}\label{eq:Bn}
B_{z,n}(\theta):=\overline{E}^Q_{0,z}\big( F_n(\theta)\,;\,\sigma=n\big),
\end{equation}
then by the arguments above, for any $\P \in \mathcal{P}_\kappa$, $n \geq 1$ and $z \in \bbV_d$, whenever $|\theta| \vee \mathrm{dis}(\P) < \gamma_{1} \wedge \frac{1}{2}$ we have \begin{equation}\label{eq:renewal1}
A_{z,n}(\theta)\leq \e^{h(1/2)} +\sum_{k=1}^{n}B_{z,k}(\theta)\sup_{z'\in \bbV_{d}}A_{z',n-k}(\theta).
\end{equation}
The next lemma will be crucial to conclude the proof.

\begin{lemma}\label{lemma8}
	There exists $\gamma_{3}=\gamma_3(y,d,\kappa) > 0$ such that, for any $\P \in \mathcal{P}_\kappa$, whenever $|\theta| \vee \mathrm{dis}(\mathbb{P}) < \gamma_3$,\begin{equation*}
	B:=\sup_{z\in \bbV_{d}}\sum_{n=1}^{\infty}B_{z,n}(\theta)<1.
	\end{equation*}
\end{lemma}

\noindent{\bf Completing proof of Proposition \ref{prop:norm2} (Assuming Lemma \ref{lemma8}):} By \eqref{eq:renewal1}, if we fix $N \geq 1$ then for any $n \leq N$ we have 
\begin{equation*}
A_{z,n}(\theta)\leq \e^{h(1/2)}+\Big(\sup_{m \leq N, z\in \bbV_{d}}A_{z,m}(\theta)\Big)\sum_{k=1}^{N}B_{z,k}(\theta),
\end{equation*} so that, upon taking suprema, we find
\begin{equation*}
\Big(1-\sum_{k=1}^{N}B_{z,k}(\theta)\Big)	\sup_{n \leq N, z\in \bbV_{d}}A_{z,n}(\theta)\leq \e^{h(1/2)}.
\end{equation*} Hence, whenever $|\theta| \vee \mathrm{dis}(\P) < \gamma_3 \wedge \gamma_1 \wedge \frac{1}{2}=:\gamma_2$, letting $N\to \infty$ we conclude by Lemma \ref{lemma8} that $A\leq \frac{\e^{h(1/2)}}{(1-B)}<\infty$ and thus Proposition \ref{prop:norm2} follows. \qed

\smallskip 

Hence, it only remains to prove Lemma \ref{lemma8}. 

\subsection{Proof of Lemma \ref{lemma8}.} 
We will need the aid of three additional lemmas. Before stating these, we introduce $B_{z,n}^*(\theta)$, the zero-disorder version of $B_{z,n}(\theta)$, given by the formula 
\[
B^*_{z,n}(\theta):=\overline{E}^Q_{0,z}\left(\e^{\langle \theta,  X_{L_{n}}+(\widetilde{X}_{\widetilde{L}_n}-z)\rangle-\overline{\Lambda}^*_{a}(\theta)(L_{n}+\widetilde{L}_n)}\,;\,\sigma=n\right),
\] where $\overline{\Lambda}^*_a(\theta):=\lim_{n\to\infty}\frac{1}{n}\log E^Q_{0}(\e^{\langle \theta,X_{n}\rangle})$ (note that this limit exists by Corollary \ref{cor:exist} applied to the particular case of zero-disorder environmental laws). The three additional lemmas we need are then the following:

\begin{lemma}\label{lemma:cot1} Given $\kappa>0$, there exists $\delta=\delta(y,d,\kappa)>0$ such that, for any $\P\in \mathcal{P}_\kappa$, 
	\[
	\sup_{z \in \bbV_{d}} \sum_{n=1}^\infty B^*_{z,n}(0) = \sup_{z \in \bbV_{d}} \overline{Q}_{0,z}(\sigma< \infty) < 1 - \delta.
	\]
\end{lemma}

\begin{lemma}\label{lemma:cot3} Given $\kappa>0$, there exist $\gamma_4=\gamma_4(y,d,\kappa),K_0=K_0(y,d,\kappa)>0$ such that 
	\[
	\sum_{n=1}^\infty  \left[\sup_{\P \in \mathcal{P}_\kappa(\gamma_4)\,,\,|\theta| < \gamma_4\,,\,z \in \bbV_{d}} B_{z,n}(\theta)\right] \leq\ K_0,
	\] where $\mathcal{P}_\kappa(\gamma_4):=\{ \P \in \mathcal{P}_\kappa : \mathrm{dis}(\P)<\gamma_4\}$.
\end{lemma}

\begin{lemma}\label{lemma:cot2} For every $n \geq 1$ and $\eta > 0$ there exists $\gamma_5=\gamma_5(y,n,\eta)>0$ such that, for any $\P \in \mathcal{P}_\kappa$, whenever $\mathrm{dis}(\P) < \gamma_5$ one has
	\[
	\sup_{|\theta|< \gamma_5\,,\,z \in \bbV_{d}} \big[B_{z,n}(\theta) - B^*_{z,n}(0)\big] < \eta.
	\]	
\end{lemma}

Proofs of Lemma \ref{lemma-cot1} - Lemma \ref{lemma:cot2} span Section \ref{sec:cot2} - Section \ref{sec:cot1}.
Assuming these, let us first complete 

\noindent{\bf Proof of Lemma \ref{lemma8} (assuming Lemma \ref{lemma:cot1}-Lemma \ref{lemma-cot2}):}

	Take $\delta=\delta(y,d,\kappa)>0$ as in Lemma \ref{lemma:cot1}. Since $B^*_{z,n}(\theta) \geq 0$, by Lemma \ref{lemma:cot3} there exists $\gamma_4=\gamma_4(y,d,\kappa)>0$ and $N=N(y,d,\kappa,\delta)\geq 1$ such that, for any $\P \in \mathcal{P}_\kappa$, if $\mathrm{dis}(\P) < \gamma_4$ then
	\begin{equation}\label{eq:cot1}
	\sum_{n  > N} \big(\sup_{|\theta| < \gamma_4\,,\,z \in \bbV_{d}} [B_{z,n}(\theta)-B^*_{z,n}(0)] \big) \leq \sum_{n > N} \big(\sup_{|\theta| < \gamma_4\,,\,z \in \bbV_{d}} B_{z,n}(\theta)\big) < \frac{\delta}{4}.
	\end{equation} Furthermore, by Lemma \ref{lemma:cot2} there exists $\gamma_5=\gamma_5(y,d,\kappa,N,\delta) > 0$ such that, for any $\P \in \mathcal{P}_\kappa$, whenever $\mathrm{dis}(\P) < \gamma_5$ we have
	\begin{equation} \label{eq:cot2}
	\sum_{n=1}^N  \sup_{|\theta| < \gamma_5\,,\,z \in \bbV_{d}} [B_{z,n}(\theta)-B^*_{z,n }(0)] < \frac{\delta}{4}.
	\end{equation} Combined with \eqref{eq:cot1} and \eqref{eq:cot2}, Lemma \ref{lemma:cot1} then yields the bound
	\[
	B \leq \sup_{z \in \bbV^d} \sum_{n=1}^\infty  B^*_{z,n}(0) +  \sum_{n=1}^\infty \sup_{|\theta|< \gamma_2\,,\,z \in \bbV_{d}} \big[B_{z,n}(\theta) - B^*_{z,n }(0)\big] < 1-\frac{\delta}{2}
	\] for any $\P \in \mathcal{P}_\kappa$ such that $\mathrm{dis}(\P)< \gamma_3:= \gamma_4 \wedge \gamma_5$. 
\qed

\subsection{Proof of of Lemma \ref{lemma:cot2}.}\label{sec:cot2}
 For $z \in \bbV_d$ and $n\geq1$, by H\"{o}lder's inequality we have
 $$
 \begin{aligned}
	B_{z,n}(\theta) - B^*_{z,n}(0)= \overline{E}^Q_{0,z}\big(F_n(\theta) -1 \,;\,\sigma=n\big)
	&\leq\big[\overline{E}^Q_{0,z}\big(\big( F_n(\theta)-1\big)^2\big)\big]^{\tfrac{1}{2}} \big[\overline{Q}_{0,z}(\sigma=n)\big]^{\tfrac 12}\\
	& \leq \big[\overline{E}^Q_{0,z}\big(\big( F_n(\theta)\big)^2\big) + 1\big]^{\tfrac 12} \big[\overline{Q}_{0,z}(\sigma=n)\big]^{\tfrac 12}.
	\end{aligned}
	$$
	Now, on the one hand, by Remark \ref{rem:bound}, the bounds $I_n \leq L_{n}\leq \tau_{n}$ and the renewal structure, whenever $|\theta| \vee h(\mathrm{dis}(\P)) < \frac{\gamma_0}{16}$, where $\gamma_0$ is the constant from Proposition~\ref{prop:exp_moments}, we have that
	\begin{equation}\label{eq:db1}
	\overline{E}^Q_{0,z}\big(\big( F_n(\theta)\big)^2\big) \leq \overline{E}^Q_{0,z}\left(\e^{4(|\theta|+h(\mathrm{dis}(\P)))\tau_{n}}\right)=\left[\overline{E}^Q_{0}\left( \e^{4(|\theta|+h(\mathrm{dis}(\P)))\tau_1}\right)\right]^{n} \leq \left[\frac{2}{\overline{c}}\right]^{n}
	\end{equation} 
	where $\overline{c} >0$ is the constant from Lemma \ref{lemma:pos}. On the other hand, by the nature of renewal times, on the event that $\sigma=n$ there exist some $k \in \{1,\cdots, L_{n}\}$ and $k' \in \{1,\dots,\widetilde{L}_{n}\}$ such that $X_k=\widetilde{X}_{k'}$. In particular, it follows that
	\begin{align}
	\overline{Q}_{0,z}(\sigma=n)  &\leq \overline{Q}_0\big( \sup_{1 \leq k \leq L_{n}} |X_k| \geq \frac{|z|}{2}\big) + \overline{Q}_z\big( \sup_{1 \leq k' \leq \widetilde{L}_{n}} |X_{k'}-z| \geq \frac{|z|}{2}\big)\nonumber \\
	&=2 \overline{Q}_0\big(\sup_{1 \leq k \leq L_{n}} |X_k| \geq \frac{|z|}{2}\big) \nonumber\\
	&\leq 2\overline{Q}_0\big( \tau_{n} \geq \frac{|z|}{2}\b) \leq 4\frac{\overline{E}^Q_0(\tau_{n})}{|z|} = 4\frac{[\overline{E}^Q_0(\tau_{1})]^{n}}{|z|}. \label{eq:db2}
	\end{align} 
	From \eqref{eq:db1} and \eqref{eq:db2}, using Lemma \ref{lemma:pos} and Proposition \ref{prop:exp_moments} it is straightforward to check that there exists $R_0=R_0(y,n,\eta)>0$ such that if $\mathrm{dis}(\P)<h^{-1}(\frac{\gamma_0}{16})$ then
	\begin{equation}\label{esti-Lemma}
	\sup_{|\theta|<\frac{\gamma_0}{16}\,,\,|z| > R_0} [B_{z,n}(\theta) - B^*_{z,n}(0)] < \eta.
	\end{equation}
	Finally, by an argument similar to the one used for \eqref{eq:db1}, Remark \ref{rem:bound} and the mean value theorem together yield that
	$|B_{z,n}(\theta) - B^*_{z,n}(0)| \leq 2(|\theta|+h(\mathrm{dis}(\P)))\overline{E}^Q_{0,z}\big(\tau_{n}\e^{2(|\theta|+h(\mathrm{dis}(\P)))\tau_n }\big)$ for any fixed $z \in \Z^d$. In particular, by Lemma \ref{lemma:pos} and Proposition \ref{prop:exp_moments} it follows that for any $R > 0$ there exists $\gamma_R = \gamma_R(y,n,R,\eta)>0$ such that if $|\theta| \vee \mathrm{dis}(\P) < \gamma_R$ then $\sup_{|\theta|<\gamma_R\,,\,|z| \leq R} [B_{z,n}(\theta) - B^*_{z,n}(0)] < \eta$. Then \eqref{esti-Lemma} yields the result with $\gamma_{5}:=h^{-1}(\frac{\gamma_0}{16}) \wedge \gamma_{R_0}$.

\subsection{Proof of Lemma \ref{lemma:cot3}.}\label{sec:cot3}
Next, we prove Lemma \ref{lemma:cot3}. If we set $\psi:=\sup\{n\in \mathcal{L}:n\leq \tau\}$ then, similarly to \eqref{eq:secondterm}, we can decompose 
	\begin{align}
	B&_{z,n}(\theta)= \sum_{j=0}^{n-1}\overline{E}^Q_{0,z}\Big( F_n(\theta)\,;\,\sigma=n\,,\,\psi=j\Big)\nonumber\\
	&\hspace{-0.2cm}=\sum_{j=0}^{n-1}\sum_{z'\in \bbV_{d}}\overline{E}^Q_{0,z}\Big(F_j(\theta)\,;\ \widetilde{X}_{\widetilde{L}_{j}}-X_{L_{j}}=z'\,,\,\psi=j\Big)\overline{E}^Q_{0,z'}\Big( F_{n-j}(\theta)\,;\,n-j=\inf\{k\in \mathcal{L}:k>0\} >\tau\Big)\nonumber\\
	&\hspace{-0.2cm}\leq \sum_{j=0}^{n-1}\bigg[\sup_{z'\in \bbV_{d}}\overline{E}^Q_{0,z}\Big(F_j(\theta)\,;\,,\widetilde{X}_{\widetilde{L}_{j}}-X_{L_{j}}=z'\,\,,\psi=j\Big)\bigg]\sum_{z'\in \bbV_{d}}D_{n-j,z'}(\theta), \label{eq:lastdisplay2}
	\end{align} where, for $n \geq 1$ and $z' \in \bbV_d$, we write
	\begin{equation}\label{eq:Dn}
	D_{n,z'}(\theta):=\overline{E}^Q_{0,z'}\Big(F_n(\theta)\,;\,n=\inf\{k\in \mathcal{L}:k>0\}>\tau\Big).
	\end{equation}
	
	Note that $\psi=j$ implies that $I_j \leq 1$ so that, recalling the random walk $Y^{(\theta)}$ with law $\widehat{P}^{(\theta)}_0$ defined in the proof of Proposition \ref{prop:norm1}, if we write $\widehat{P}^{(\theta)}_{0,0}:=\widehat{P}^{(\theta)}_{0}\times\widehat{P}^{(\theta)}_{0}$ then for any $j \geq 1$  we have 
	$$
	\begin{aligned}
	\overline{E}^Q_{0,z}\Big(F_j&(\theta)\,;\,\widetilde{X}_{\widetilde{L}_{j}}-X_{L_{j}}=z'\,,\,\psi=j\Big) \leq \e^{h(\mathrm{dis}(\P))}\overline{E}^Q_{0,z}\Big( \phi_j(\theta)\widetilde{\phi}_j(\theta)\,;\,\widetilde{X}_{\widetilde{L}_{j}}-X_{L_{j}}=z'\,,\,j\in \mathcal{L}\Big) \\
	&=\e^{h(\mathrm{dis}(\P))}\widehat{P}^{(\theta)}_{0,0}\Big(\exists k,m:\langle Y_{k},\ell\rangle=j\,,\,\widetilde{Y}_{m}-Y_{k}=z'-z\Big) \\
	& \leq \e^{h(\mathrm{dis}(\P))}\sum_{\langle x,\ell\rangle=j}\widehat{P}^{(\theta)}_{0}\Big(\exists k : \langle Y_{k},\ell\rangle=x\Big) \widehat{P}^{(\theta)}_{0}\Big(\exists m:\langle \widetilde{Y}_{m},\ell\rangle=x+z'-z\Big).
	\end{aligned}
	$$
	Thus, 
	\begin{align}
	&\overline{E}^Q_{0,z}\Big(F_j(\theta)\,;\,\widetilde{X}_{\widetilde{L}_{j}}-X_{L_{j}}=z'\,,\,\psi=j\Big)\nonumber\\
		&\leq \e^{h(\mathrm{dis}(\P))}\big[\sup_{\langle x,\ell\rangle=j}\widehat{P}^{(\theta)}_{0}\big(\exists k : \langle Y_{k},\ell\rangle=x\big)\big]\sum_{\langle x,\ell\rangle=j}\widehat{P}^{(\theta)}_{0}\big(\exists m:\langle \widetilde{Y}_{m},\ell\rangle=x+z'-z\big)\label{eq:lasteq3}\\
	&= \e^{h(\mathrm{dis}(\P))}\big[\sup_{\langle x,\ell\rangle=j}\sum_{k\in \N}\widehat{P}^{(\theta)}_{0}\Big(\langle Y_{k},\ell\rangle=x\Big)\big]\widehat{P}^{(\theta)}_{0}( \exists\,m : \langle\widetilde{Y}_m,\ell\rangle = j)\nonumber\\
	& \leq \e^{h(\mathrm{dis}(\P))}\sup_{\langle x,\ell\rangle=j}\sum_{k\in \N}\mu_{k}^{(\theta)}(x),\label{eq:lasteq} 
	\end{align} where $\mu^{(\theta)}$ is as in \eqref{eq:Probtheta} and, given any probability measure $\mu$, $\mu_k$ denotes its $k$-fold convolution. Observe that for $j=0$ we obtain directly from \eqref{eq:lasteq3} the upper bound $\e^{h(\mathrm{dis}(\P))}$.
	
	Now, in the proof of \cite[Theorem 5.1]{BS02} it is shown that, whenever $d \geq 4$, given any $c_1,c_2,c_3 > 0$ there exists $K_1=K_1(d,c_1,c_2,c_3)>1$ such that for any $j \geq 1$
	\begin{equation} \label{eq:boundbz}
	\sup_{\langle x,\ell\rangle=j}\sum_{k\in \N}\mu_{k}(x) \leq \frac{K_1}{(1+j)^{(d-1)/2}}
	\end{equation} holds uniformly over all probability measures $\mu$ on $\Z^d$
	satisfying
	\begin{enumerate}
		\item [C1.] $\sum_{x \in \Z^d} \mu(x)\e^{c_1 |x|} \leq 2$,
		\item [C2.] $\Sigma_\mu \geq c_2 I_d$, where $I_d$ denotes the $d\times d$ identity matrix,
		\item [C3.] $|\sum_{x \in \Z^d} \langle x,\ell\rangle \mu(x)| > c_3$.
	\end{enumerate} More precisely, it is shown that for any measure $\mu$ satisfying these conditions and $k \in \N$ one has the estimate
	\[
	\mu_k(x) \leq C(\varphi_k^{(1)}(x) + \varphi_k^{(2)}(x))
	\] for some constant $C=C(d,c_1,c_2,c_3)>0$, where
	\[
	\sum_{k \in \N} \varphi_k^{(1)}(x) \leq \frac{K'_1}{(1+|x|)^{-(d-1)/2}} \hspace{1cm}\text{ and }\hspace{1cm}\sum_{k \in \N} \varphi_k^{(2)}(x) \leq K''_1\e^{ -\delta|x|}
	\] for some constants $\delta,K'_1,K''_1>0$ depending only on $d,c_1,c_2$ and $c_3$.
	
	Thus, to bound \eqref{eq:lasteq} we will show that there exists $\nu=\nu(y,d,\kappa) > 0$ such that, for any $\P \in \mathcal{P}_\kappa$, whenever $|\theta| \vee \mathrm{dis}(\P) < \nu$ the measure $\mu^{(\theta)}$ satisfies (C1)-(C2)-(C3) above for some $c_1,c_2,c_3>0$ depending only on $y,d$ and $\kappa$. Indeed, by the same type of argument leading to \eqref{eq:db1}, we have
	\[
	\Big|\sum_{x \in \Z^d} \mu^{(\theta)}(x)\e^{c_1 |x|}-1\Big| \leq (2|\theta|+h(\mathrm{dis}(\P))+c_1)\overline{E}^Q_0(\tau_1 \e^{(2|\theta|+h(\mathrm{dis}(\P))+c_1)\tau_1})
	\] so that, by Lemma \ref{lemma:pos} and Proposition \ref{prop:exp_moments}, there exists $\nu_1=\nu_1(y)>0$ such that if $c_1>0$ is taken small enough (depending only on $y$) then (C1) holds when $|\theta|\vee \mathrm{dis}(\P) < \nu_1$. On the other hand, since $\langle X_{\tau_1}-X_0,\ell\rangle \geq +1$ by definition of $\tau_1$, it follows that
	\[
	\big|\sum_{x \in \Z^d} x\mu^{(\theta)}(x)\big| \geq \overline{E}^Q_0(\phi_1(\theta)\langle X_{\tau_1},\ell\rangle) \geq  \overline{E}^Q_0(\phi_1(\theta))=1
	\] and so (C3) is satisfied with $c_3:=1$. Finally, to check (C2) we first notice that by \eqref{eq:gradient_identity} and \eqref{eq:hessian_formula},
	\[
	\Sigma_{\mu^{(\theta)}}= H_a(\theta) \overline{E}^Q_0\big(\tau_{1}\phi_1(\theta)\big).
	\] Since $\Sigma_{\mu^{(\theta)}}$ is a positive definite matrix whenever $|\theta|\vee \mathrm{dis}(\P) <\gamma_1$ by Proposition \ref{prop:hessian}, to obtain (C2) it will suffice to show that there exists $\nu_2=\nu_2(y,d,\kappa)>0$ such that if $|\theta|\vee \mathrm{dis}(\P) < \nu_2$ then
	\begin{equation}\label{eq:sing}
	\inf_{\P \in \mathcal{P}_\kappa(\nu_2)\,,\,|\theta|<\nu_2}\sigma_{\min}(\Sigma_{\mu^{(\theta)}})\geq c_2
	\end{equation} for some constant $c_2>0$ depending only on $y,d$ and $\kappa$, where $\sigma_{\min}(A)$ above denotes the smallest singular value of a matrix $A$. Since $\overline{E}^Q_0(\tau_1G(1,\theta))\geq 1$ and $1/\sigma_{\min}(A) = \lVert A^{-1}\rVert_2 \leq \sqrt{d}\lVert A^{-1}\rVert$ for any invertible $A \in \R^{d\times d}$, where $\lVert\cdot\rVert_2$ and $\lVert\cdot\rVert$ denote the operator $2$-norm and $1$-norm respectively, we see that \eqref{eq:sing} will hold if we show that for some $\nu_2=\nu_2(y,d,\kappa)>0$ we have
	\[
	\sup_{\P \in \mathcal{P}_\kappa(\nu_2)\,,\,|\theta|<\nu_2} \lVert (H_a(\theta))^{-1}\rVert < \infty
	\] and take $c_2:= (\sqrt{d}\sup_{\P \in \mathcal{P}_\kappa(\nu_2)\,,\,|\theta|<\nu_2} \lVert (H_a(\theta))^{-1}\rVert)^{-1}$. Using once again the identity $A^{-1} - B^{-1} = A^{-1}(B-A)B^{-1}$ for invertible matrices $A,B \in \R^{d\times d}$, we have
	\begin{equation}\label{eq:uid}
	\lVert (H_a(\theta))^{-1} - (H_a(0))^{-1} \rVert \leq  \lVert (H_a(\theta))^{-1}\rVert\lVert H_a(\theta) - H_a(0)\rVert \lVert (H_a(0))^{-1} \rVert.
	\end{equation} But then, by the proof of Proposition \ref{prop2} there exist $\nu_2=\nu_2(y,d,\kappa), c=c(y,d,\kappa)>0$ such that 
	\[
	\sup_{\P \in \mathcal{P}_\kappa(\nu_2)} \lVert (H_a(0))^{-1} \rVert \leq c \hspace{1cm}\text{ and }\hspace{1cm}\sup_{\P \in \mathcal{P}_\kappa(\nu_2)\,,\,|\theta|<\nu_2} \lVert H_a(\theta) - H_a(0)\rVert < \frac{1}{2c},
	\] which by \eqref{eq:uid} and the triangle inequality implies that
	\[
	\sup_{\P \in \mathcal{P}_\kappa(\nu_2)\,,\,|\theta|<\nu_2} \lVert (H_a(\theta))^{-1}\rVert \leq 2c <\infty
	\] and so (C2) follows. Thus, we see that for $\nu:=\nu_1 \wedge \nu_2 \wedge \frac{1}{2}$ we have by \eqref{eq:lastdisplay2}, \eqref{eq:lasteq} and \eqref{eq:boundbz}
	\begin{align*}
	\big[\sup_{\P \in \mathcal{P}_\kappa(\nu)\,,\,|\theta| <\nu\,,\,z \in \bbV_{d}} B_{z,n}(\theta)\big] \leq \e^{h(1/2)}K_1 \sum_{j=0}^{n-1} \frac{1}{(1+j)^{(d-1)/2}} \sum_{z'\in \bbV_{d}} \sup_{\P \in \mathcal{P}_\kappa(\nu)\,,\,|\theta| <\nu}D_{n-j,z'}(\theta)
	\end{align*} so that
	\begin{equation*}
	\sum_{n=1}^\infty \big[ \sup_{\P \in \mathcal{P}_\kappa(\nu)\,,\,|\theta| <\nu\,,\,z \in \bbV_{d}} B_{z,n}(\theta)\big]\leq \e^{h(1/2)}K_1 
	\sum_{j=0}^{\infty} \frac{1}{(1+j)^{(d-1)/2}}\sum_{n=1}^{\infty}\sum_{z\in \bbV_{d}}\sup_{\P \in \mathcal{P}_\kappa(\nu)\,,\,|\theta| <\nu}D_{n,z}(\theta).
	\end{equation*}
	
	The proof of Lemma \ref{lemma:cot3} will then be complete once we prove the result stated below.\qed

\begin{lemma}\label{lemma14}
	There exist $\gamma_6=\gamma_6(y),K'=K'(y,d)>0$ such that 
	\begin{equation*}
	\sum_{n=1}^{\infty}\sum_{z\in \bbV_{d}}\sup_{\P \in \mathcal{P}_\kappa(\gamma_6)\,,\, \abs{\theta}<\gamma_6}D_{n,z}(\theta)\leq K'.
	\end{equation*}
\end{lemma}
\begin{proof} By Cauchy-Schwarz inequality,
	\begin{equation}\label{eq:proddesc}
	D_{n,z}(\theta)\leq \Big(\overline{E}_{0,z}^{Q}\big( (F_n(\theta))^2\big)\Big)^{1/2} \Big(\overline{P}_{0,z}^{Q}(n=\inf\{k\in \mathcal{L}:k>0\})\Big)^{1/4}\Big(\overline{P}_{0,z}^{Q}(n>\tau)\Big)^{1/4}. 
	\end{equation}
	As in \eqref{eq:db1}, the first factor on the right-hand side of \eqref{eq:proddesc} can be bounded from above by 
	\begin{equation}
	\label{eq:firstfactor}
	\left[\overline{E}^Q_{0}\left( \e^{4(|\theta|+h(\mathrm{dis}(\P)))\tau_1}\right)\right]^{n/2} \leq \left[\overline{E}^Q_{0}\left( \e^{\frac{\gamma_0}{2}\tau_1}\right)\right]^{\frac{4(|\theta|+h(\mathrm{dis}(\P)))}{\gamma_0}n}\leq \e^{\log(2/\overline{c})\frac{4(|\theta|+h(\mathrm{dis}(\P)))}{\gamma_0}n}
	\end{equation} whenever $|\theta| \vee h(\mathrm{dis}(\P)) < \frac{\gamma_0}{16}$, with $\gamma_0$ as in Proposition \ref{prop:exp_moments}, by Jensen's inequality.
	
	On the other hand, to deal with the third factor we notice that if $z \neq 0$ then whenever $n>\tau$ then $X_{i} = \widetilde{X}_{j}$ for some $1 \leq i \leq \tau_n$ and $1 \leq j \leq \widetilde{\tau}_n$ so that, in particular, we must have $\tau_n \vee \widetilde{\tau}_n \geq \frac{|z|}{2}$. Then, using the inequality $(a_{1}+\cdots+a_{n})^{m}\leq n^{m-1}(a_{1}^{m}+\cdots+a_{n}^{m})$, valid for positive $(a_{i})_{1\leq i\leq n}$ and $m\geq 1$, by the union bound we obtain
	\begin{align*}
	\overline{P}_{0,z}^{Q}(n>\tau)&\leq 2\overline{P}_{0}^{Q}\Big(\tau_{n}\geq \frac{|z|}{2}\Big)\leq 2\Big(\frac{2}{|z|}\Big)^{4d+1}\overline{E}_{0}^{Q}(\tau_{n}^{4d+1})\leq 2\Big(\frac{2}{|z|}\Big)^{4d+1}n^{4d}\overline{E}_{0}^{Q}(\tau_{1}^{4d+1}) 
	\end{align*} From this, by the trivial bound $\overline{P}^Q_{0,0}(n > \tau) \leq 1$ and Proposition \ref{prop:exp_moments} we conclude that there exists $K'_1=K'_1(d,y)>0$ such that, for any $n \geq 1$ and $z \in \bbV_d$,
	\begin{equation}
	\label{eq:secondfactor}
	\Big(\overline{P}_{0,z}^{Q}(n>\tau)\Big)^{1/4} \leq K'_1 n^d (1 \vee |z|)^{-\big(d+\tfrac 14\big)}.
	\end{equation}
	
	Finally, to control the middle factor in the right-hand side of \eqref{eq:proddesc}, we will show that there exist $c=c(y),K_2'=K_2'(y)>0$ such that, for any $\P \in \mathcal{P}_\kappa$,
	\begin{equation}
	\label{eq:expcontrol}
	\sup_{z\in \bbV_{d}}\overline{E}_{0,z}^{Q}(\e^{4c\lambda^*})< (K_2')^4
	\end{equation} where $\lambda^*:= \inf\{k\in \mathcal{L}:k>0\}$, so that
	\begin{equation}
	\label{eq:expcontrol2} \Big(\overline{P}_{0,z}^{Q}(n=\inf\{k\in \mathcal{L}:k>0\})\Big)^{1/4} \leq \Big(\overline{P}_{0,z}^{Q}(\lambda^* \geq n)\Big)^{1/4} \leq K_2'\e^{-cn}.
	\end{equation} To this end, for $m \geq 0$ define 
	\[
	\beta(m):=\inf\{ n \geq L_m : \langle X_n, \ell \rangle < m\}\hspace{1cm}\text{ and }\hspace{1cm}R(m):=\sup\{ \langle X_n,\ell\rangle : L_m \leq n < \beta(m)\},
	\] together with the corresponding quantities $\widetilde{\beta}(m)$,$\widetilde{R}(m)$ for $\widetilde{X}$ and consider the sequence $(\lambda_j)_{j \geq 1}$ defined inductively by first taking $\lambda_1:=1$ and then setting 
	\[ 
	\lambda_{j+1}= \begin{cases}
	R(\lambda_j) \wedge \widetilde{R}(\lambda_j) + 1 & \text{ if }\lambda_j<\infty\\ 
	\infty & \text{ if }\lambda_j=\infty.
	\end{cases}
	\] It is not hard to check that $\lambda^*=\sup\{ \lambda_j : \lambda_j < \infty\}$. We will use this representation of $\lambda^*$ to estimate its exponential moments and show \eqref{eq:expcontrol}. In order to do this, let us first observe that if we define $\lambda:=R(0) \wedge \widetilde{R}(0) +1$ then, 
	for any $z \in \bbV_d$ and $\hat{c} \in (0,\gamma_0)$ (with $\gamma_0$ as in Proposition \ref{prop:exp_moments}), we have by H\"{o}lder's inequality that
	\[
	E^Q_{0,z}(\e^{\hat{c}\lambda}\,;\,\lambda < \infty) \leq \left[ E^Q_{0,z}(\e^{\gamma_0\lambda}\,;\,\lambda < \infty)\right]^{\tfrac{\hat{c}}{\gamma_0}} \left[ Q_{0,z}(\lambda < \infty)\right]^{1-\tfrac{\hat{c}}{\gamma_{0}}}.
	\] Since $Q_0$-a.s. we have $R(0)+1 \leq \tau_1$ on the event that $\beta_0 < \infty$ (observe that $\beta_0=\beta(0)$ $Q_0$-a.s.), then by Proposition \ref{prop:exp_moments}
	\[
	E^Q_{0,z}(\e^{\gamma_0\lambda}\,;\,\lambda < \infty) \leq E^Q_{0,z}(\e^{\gamma_0(R(0)+1)}\,;\, \beta_0 < \infty) + E^Q_{0,z}(\e^{\gamma_0(\widetilde{R}(0)+1)}\,;\,\widetilde{\beta}_0 < \infty) \leq 2 E^Q_{0,z}(\e^{\gamma_0\tau_1}) \leq 4.
	\] On the other hand, by Lemma \ref{lemma:pos} we have 
	$
	Q_{0,z}(\lambda < \infty) \leq Q_{0,z}(\beta_0 < \infty \text{ or }\widetilde{\beta}_0 < \infty) = 1 - (Q_0(\beta_0=\infty))^2 < 1 - \overline{c}^2.
	$ It follows that for some $\hat{c}=\hat{c}(y) \in (0,\gamma_0)$ sufficiently small we have
	\[
	\sup_{z \in \bbV_d}E^Q_{0,z}(\e^{\hat{c}\lambda}\,;\,\lambda < \infty)  \leq 1-\frac{\overline{c}^2}{2}.
	\] With this, using the Markov property and translation invariance, for $z \in \bbV_d$ we may compute
	\begin{align*}
	E^Q_{0,z}(\e^{\hat{c}\lambda^*}) &= \sum_{j=1}^\infty E^Q_{0,z}(\e^{\hat{c}\lambda_j}\,;\,\lambda^*=\lambda_j) \leq  \sum_{j=1}^\infty  E^Q_{0,z}(\e^{\hat{c}\lambda_j}\,;\,\lambda_j<\infty)\\
	& = \sum_{j=1}^\infty \sum_{z' \in \bbV_d} E^Q_{0,z}(\e^{\hat{c}\lambda_{j-1}}\,;\,\lambda_{j-1}<\infty, \widetilde{X}_{\widetilde{L}_{\lambda_{j-1}}}-X_{L_{\lambda_j}}=z')E^Q_{0,z'}(\e^{\hat{c}\lambda}\,;\,\lambda<\infty)\\
	& \leq \sum_{j=1}^\infty E^Q_{0,z}(\e^{\hat{c}\lambda_{j-1}}\,;\,\lambda_{j-1}<\infty)(1-\tfrac{\overline{c}^2}{2}),
	\end{align*} so that by induction we conclude that
	$
	\sup_{z \in \bbV^d} E^Q_{0,z}(\e^{\hat{c}\lambda^*}) \leq \e^{ \hat{c}} \sum_{j=1}^\infty(1-\tfrac{\overline{c}^2}{2})^{j-1} = \frac{2\e^{\hat{c}}}{\overline{c}^2}, 
	$and so \eqref{eq:expcontrol} follows. 
	Gathering \eqref{eq:firstfactor}, \eqref{eq:expcontrol2} and \eqref{eq:secondfactor}, from \eqref{eq:proddesc} we see that if $\gamma_6>0$ is chosen sufficiently small so that $|\theta| \vee h(\mathrm{dis}(\P)) < \frac{\gamma_0}{16}$ and $\log(2/\overline{c})\frac{4(|\theta|+h(\mathrm{dis}(\P)))}{\gamma_0} < \tfrac{c}{2}$ with $c$ as in \eqref{eq:expcontrol2} (which can be done depending only on $y$), then  
	\[
	\sum_{n=1}^{\infty}\sum_{z\in \bbV_{d}}\sup_{\P \in \mathcal{P}_\kappa(\gamma_6)\,,\, \abs{\theta}<\gamma_6}D_{n,z}(\theta) \leq K_1'K_2' \big[	\sum_{n=1}^{\infty} n^d \e^{-\tfrac{c}{2}n}\big]\big[\sum_{z\in \bbV_{d}} (1 \vee |z|)^{-\big(d+\tfrac{1}{4}\big)}\big]=:K'< \infty,
	\] which completes the proof.
\end{proof}

\subsection{Proof of Lemma \ref{lemma:cot1}.}\label{sec:cot1}We finish by giving the proof of Lemma \ref{lemma:cot1}. We first notice that there exist constants $\eta_1,\eta_2,\eta_3>0$, all depending only on $y,d$ and $\kappa$ such that, for any $\P \in \mathcal{P}_\kappa$, 
	\begin{enumerate}
		\item [D1.] $Q_0(\beta_0=\infty) > \eta_1$,
		\item [D2.] $E^Q_0(\tau_1^9)< \eta_2$,
		\item [D3.] $\sup_{z \in \Z^d} \overline{Q}_0(X_{\tau_n}=z) \leq \eta_3 n^{-d/2}$ for any $n \geq 1$.
	\end{enumerate} Indeed, (D1)-(D2) follow immediately from Lemma \ref{lemma:pos} and Proposition \ref{prop:exp_moments}, respectively. To check (D3), note that for any $\P \in \mathcal{P}_\kappa$ the law $\mu^*$ of $X_{\tau_1}$ under $\overline{Q}_0$ satisfies conditions (C1)-(C2)-(C3) in the proof of Lemma \ref{lemma:cot3} for some constants $c_1,c_2,c_3>0$ which depend only on $y,d$ and $\kappa$. Indeed, this follows from the proof of Lemma \ref{lemma:cot3} upon noticing that $\mu^*$ coincides with $\mu^{(0)}$ for the zero-disorder law $\P_\alpha$ with marginals $\alpha \in \mathcal{M}_1^{(\kappa)}(\bbV)$. By \cite[Eq. 5.5]{BS02}, this gives (D3) for some $\eta_3$ depending only on $y,d$ and $\kappa$.
	
	Under these conditions, since $\sigma < \infty$ implies that the two walks need to intersect at a time other than zero, by essentially repeating the proofs of \cite[Propositions 3.1 and 3.4]{BZ08} (but using instead the estimates in (D1)-(D2)-(D3) which are uniform over $\P \in \mathcal{P}_\kappa$), it can be shown that there exists $N=N(y,d,\kappa) \geq 1$ such that, for any $\P \in \mathcal{P}_\kappa$,
	\[
	\sup_{|z| \geq 2N} \overline{Q}_{0,z}(\sigma < \infty) \leq \frac{1}{2}.
	\]	To deal with $z \in \Z^d$ such that $|z| < 2N$, take any such $z$ together with $\e^* \in \bbV \setminus \{\ell,-\ell\}$ and assume without loss of generality that $\langle z, \e^* \rangle \geq 0$. Then consider the events 
	\[
	E_1:=\{ X_N = -Ne^*\,,\,\widetilde{X}_N = \widetilde{X}_0+Ne^*\} \qquad E_2:=\{ X_i \neq \widetilde{X}_j \text{ for all }i,j\ > N\}.
	\] Since $|z+2Ne^*| \geq 2N$ by choice of $e^*$ and on $E_1$ we have both $\langle X_i - X_0 , \ell\rangle = \langle \widetilde{X}_j - \widetilde{X}_0, \ell\rangle = 0$ and $X_i \neq \widetilde{X}_j$ for all $1 \leq i,j \leq N$, using (P1) from Lemma \ref{lemma:prop} and translation invariance, we obtain
	\[
	\overline{Q}_{0,z}(\sigma = \infty) \geq \overline{Q}_{0,z}(E_1 \cap E_2) \geq c_\kappa^{2N} \inf_{|y| \geq 2N} \overline{Q}_{0,y}(\sigma = \infty) > \frac{1}{2}c_\kappa^{2N} > 0
	\] for any $\P \in \mathcal{P}_\kappa$, so that now Lemma \ref{lemma:cot1} follows upon taking $\delta:=\frac{1}{2}c_\kappa^{2N}$.\qed

\noindent{\bf Acknowledgement.} 
The authors are very grateful to  Noam Berger, Nina Gantert and Atilla Yilmaz for very useful comments on
 an earlier version of the manuscript. The first author has been supported by ANID-PFCHA/Doctorado Nacional no. 2018-21180873. The second author is supported by the Deutsche Forschungsgemeinschaft (DFG) under Germany's Excellence Strategy EXC 2044--390685587, Mathematics M\"unster: Dynamics--Geometry--Structure. The third author has been partially supported by Fondo Nacional de Desarrollo  Cient\'ifico y Tecnol\'ogico 1180259 and Iniciativa
Cient\'\i fica Milenio. The fourth author has been supported in part at the Technion by a fellowship from the Lady Davis Foundation, the Israeli Science Foundation grants no. 1723/14 and 765/18, and by the NYU-ECNU Institute of Mathematical Sciences at NYU Shanghai. This research was also supported by a grant from the United States-Israel Binational Science Foundation (BSF), no. 2018330.


\begin{thebibliography}{WWW98}


\bibitem[Bau16]{Ba16} {\sc E. Baur}.
  \newblock{An invariance principle for a class of non-ballistic
    random walks in random environment.}
  \newblock{Probab. Theory Related Fields} {\bf 166}, 463-514 (2016).
  
\bibitem[BMRS21]{BMRS21}
{\sc R. Bazaes}, {\sc C. Mukherjee}, {\sc A. Ram\'irez} and {\sc S. Saglietti.}
\newblock{The effect of disorder on quenched and averaged large deviations for random walks in random environment: boundary behavior.}
\newblock{\it Preprint}, arXiv: 2101.04606 (2021)


\smallskip 


\bibitem[Ber08]{Ber08}
{\sc N. Berger}. 
\newblock{Limiting velocity of high-dimensional random walk in random environment.}
\newblock{\it Ann. Probab.} {\bf 36}, 
no. 2, 728-738,   (2008). 

\smallskip 

\bibitem[Ber12]{Ber12}
{\sc N. Berger}. 
\newblock{Slowdown estimates for ballistic random walk in random environment. }
{\it J. Eur. Math. Soc. (JEMS)} {\bf 14}, 
no. 1, 127-174, (2012). 


\smallskip 


 \bibitem[BB07]{BB07}
{\sc N. Berger} and {\sc M. Biskup}.
\newblock Quenched invariance principle for random walk on percolation clusters,
\newblock {\it Probab. Theory Rel. Fields}, {\bf 137}, Issue 1-2, 83-120, (2007)

 \smallskip




 \bibitem[BD14]{BD14}
 {\sc N. Berger} and {\sc J-D. Deuschel.}
 \newblock{A quenched invariance principle for non-elliptic random walk in i.i.d. balanced random environment.}
\newblock{\it Prob. Theory Related Fields}, {\bf 158}, (2014), 91-126

\smallskip






\bibitem[BMO16]{BMO16}
{\sc N. Berger}, {\sc C. Mukherjee} and {\sc K. Okamura}. 
\newblock{Quenched large deviations for random walks in percolation models including long range correlations.}
\newblock{\it Comm. Math. Phys.}, {\bf 358}, 633-673, (2018). 

\smallskip 


\bibitem[BS02]{BS02} 
{\sc E. Bolthausen} and {\sc A-S. Sznitman}.
\newblock{On the static and dynamic points of view for certain random walks in random environment.}
\newblock{\it Methods and Applications of Analysis}, {\bf 9 }(3), 245-276, (2002).

\smallskip


\bibitem[BZ08]{BZ08}
{\sc N. Berger} and {\sc O. Zeitouni}. 
\newblock{A quenched invariance principle for certain ballistic random walks in i.i.d. environments.}
\newblock{In and out of equilibrium.} {\bf 2}, 137-160, Progr. Probab., 60, Birkh\"auser, Basel, (2008). 

\smallskip 




\bibitem[BZ07]{BZ07}
{\sc E. Bolthausen} and {\sc O. Zeitouni}. 
\newblock{Multiscale analysis of exit distributions for random walks in random environments.}
\newblock{\it Prob. Theory Related Fields}, {\bf 138}, 581-645, (2007). 

\smallskip 

\bibitem[BK91]{BK91}
{\sc Random walks in asymmetric random environments.}
\newblock{\it Comm. Math. Phys.} {\bf 142}, 345?420, (1991)

\smallskip


\bibitem[CGZ00]{CGZ00}
 {\sc F. Comets}, {\sc N. Gantert} and {\sc O. Zeitouni}. 
 \newblock{Quenched, annealed and functional large deviations for one dimensional random walks in random environments.}
 \newblock{\it Prob. Theory Related Fields}, {\bf 118}, 65-114, (2000). 
 
 
  

\smallskip 

\bibitem[DZ98]{DZ98}
{\sc A. Dembo} and {\sc O. Zeitouni}. 
\newblock{Large deviation techniques and applications.}
\newblock{2nd ed. Springer, New York,  (1998).}

\smallskip 

\bibitem[DR10]{DR10}
{\sc A Drewitz} and {\sc A. Ram\'irez}.
\newblock{Asymptotic direction in random walks in random environment revisited.}
\newblock{\it Braz. J. Probab. Stat.} {\bf 24} 212-225 (2010).

\smallskip



\bibitem[GZ98]{GZ98}
{\sc N. Gantert} and {\sc O. Zeitouni}.
\newblock{Quenched sub-exponential tail-estimates for one-dimensional random walk in random environment}
{\it Comm. Math. Phys.}, {\bf 194},  177- 190 (1988).

\smallskip 

\bibitem[GKS07]{GKS07}
{\sc N. Gantert}, {\sc W. K\"onig} and {\sc Z. Shi}.
\newblock{Annealed deviations of random walk in random scenery.}
{\it Annales de l'Institut Henri Poincar\'e Prob. et Stat.,} {\bf 43},
147-176 (2007).

\smallskip

\bibitem[GZ12]{GZ12}
{\sc X. Guo} and {\sc O. Zeitouni.}
\newblock{Quenched invariance principle for random walks in balanced random environment.}
{\it Probab. Theory Related Fields.} {\bf 152}, 207-230 (2012)

\smallskip


\bibitem[GdH98]{GdH94}
{\sc A. Greven} and {\sc F. den Hollander}. 
\newblock{Large deviations for a random walk in a random environment.}
\newblock{{\it Ann. Prob.}}, {\bf 22}, 1381-1428, (1998). 

\smallskip 

\bibitem[GR20]{GR20}
{\sc E. Guerra} and {\sc A. Ram\'irez}. 
\newblock{A proof of Sznitman's conjecture about ballistic RWRE.}
{\it Comm. Pure Appl. Math}, {\bf 73}, 2087-2103 (2020). 



\smallskip 

\bibitem[KS79]{KS79}
{\sc H. Kesten} and {\sc F. Spitzer}.
\newblock{A limit theorem related to a new class of self-similar processes,}
{\it Z. Wahrsch. Verw. Geb.}{\bf 50}, 5-25 (1979).







\bibitem[KV86]{KV86}
{\sc C. Kipnis} and {\sc S.R.S.~Varadhan}. 
\newblock Limit theorem for additive functionals of reversible Markov chains and application to simple exclusions. 
\newblock {\it Comm. Math. Phys.} {\bf{104}}, 1-19, (1986). 

\smallskip 



\bibitem[KRV06]{KRV06}
{\sc E. Kosygina}, {\sc F. Rezakhanlou} and {\sc S. R. S. Varadhan}. 
\newblock{Stochastic homogenization of Hamilton-Jacobi-Bellmann equations.}
\newblock{\it Comm. Pure Appl. Math.}, {\bf 59}, 1489-1521, (2006). 

\smallskip 

\bibitem[KV08]{KV08}
{\sc E. Kosygina} and {\sc S. R. S. Varadhan}.
\newblock{Homogenization of Hamilton-Jacobi-Bellman equations with respect to time-space shifts in a stationary ergodic medium.}
\newblock{\it Comm. Pure Appl. Math.}, {\bf 61}, 816-847, (2008).

\smallskip





\bibitem[K85]{K85}
{\sc S. M. Kozlov}. 
\newblock{The averaging method and walks in inhomogeneous environments.}
\newblock{\it Uspekhi Mat. Nauk}, {\bf 242}, 61-120, (1985). 

\smallskip 

\bibitem[KP02]{KP02} 
{\sc S.G. Krantz} and {\sc H.R. Parks}.
\newblock{The Implicit Function Theorem: History, Theory and Applications.}
\newblock{ Birkhaüser, Boston (2002)}.

\smallskip


\bibitem[K12]{K12}
{\sc N. Kubota.}
\newblock{Large deviations for simple random walk on supercritical percolation clusters.}
\newblock{\it Kodai Mathematical Journal} {\bf 35} 560-575, (2012).

\smallskip 

\bibitem[L82]{L82}
{\sc G. Lawler}. 
\newblock{Weak convergence of a random walk in a random environment.}
\newblock{\it Comm. Math. Phys.} {\bf 87},  81-87, (1982). 

\smallskip 

  \bibitem[MP07]{MP07}
{\sc P. Matheiu} and {\sc A. Piatnitski}.
\newblock Quenched invariance principle for random walks on percolation clusters ,
\newblock{Proceedings of the Royal Society A.}, 463, 2287-2307, (2007).

\smallskip







\bibitem[M12]{M12}
{\sc J.-C. Mourrat}.
\newblock{Lyapunov exponents, shape theorems and large deviations for random walks in random potential.}
\newblock{\it ALEA Lat. Am. J. Probab. Math. Stat.} {\bf 9}, 165-211 (2012).




\smallskip 





\bibitem[PV81]{PV81}
{\sc G. C. Papanicolaou} and {\sc S. R. S. Varadhan}. 
\newblock{ Boundary value problems with rapidly oscillating random 
  coefficients.}
\newblock In Random fields, Vol. I, II (Esztergom, 1979), volume 27 
of Colloq. Math. Soc. Janos Bolyai, pages 835-873. North-Holland, Amsterdam, (1981). 

\smallskip 

\bibitem[PZ09]{PZ09}
{\sc J. Peterson} and {\sc O. Zeitouni}. 
\newblock{On the annealed large deviation rate function for a multi-dimensional random walk in random environment.}
\newblock{\it ALEA.} {\bf 6}, 349-368, (2009). 

\smallskip 

\bibitem[RS09]{RS09}
{\sc F. Rassoul-Agha} and {\sc T. Sepp\"al\"ainen.}
\newblock{Almost sure functional central limit theorem for ballistic random walk in
random environment. }
\newblock{\it Ann. Inst. Henri Poincar\'e Probab. Stat.}, {\bf 45}, 373-420, (2009).

\smallskip 


\bibitem[RS11]{RS11}
{\sc F. Rassoul-Agha} and {\sc T. Sepp\"al\"ainen}. 
\newblock{Process-level quenched large deviations for random walk in a random environment.}
\newblock{\it Ann. Inst. H. Poincar\'e Prob. Statist.}, {\bf 47}, 214-242, (2011). 

\smallskip 


\bibitem[RS14]{RS14}
{\sc F. Rassoul-Agha} and {\sc T. Sepp\"al\"ainen}. 
\newblock{Quenched point-to-point free energy for random walks in random potentials.}
\newblock{\it Probab. Theory Related Field.}, {\bf 158} (3-4), 711-750, (2014). 

\smallskip


\bibitem[FSY17a]{RSY17a}
{\sc F. Rassoul-Agha}, {\sc T. Sepp\"al\"ainen} and {\sc A. Yilmaz.}
\newblock{Variational formulas and disorder regimes of random walks in random potential.}
\newblock{\it Bernoulli} {\bf 23}  405-431, (2017).

\smallskip

\bibitem[FSY17b]{RSY17b}
{\sc F. Rassoul-Agha}, {\sc T. Sepp\"al\"ainen} and {\sc A. Yilmaz.}
\newblock{Averaged vs. quenched large deviations and entropy for random walk in a dynamic random environment.}
\newblock{\it Electronic J. Probab.} {\bf 22} (2017).

\smallskip 



\bibitem[R97]{R97}
{\sc R. T. Rockafellar}. 
\newblock{Convex analysis.}
\newblock{Princeton University Press}, Princeton, N. J., (1997). 

\smallskip 


\bibitem[R06]{R06}
{\sc J. Rosenbluth}. 
\newblock{Quenched large deviations for multidimensional random walks in a random environment: a variational formula.}
\newblock{\it PhD thesis}, NYU, arxiv:0804.1444v1, (2006). 

\smallskip 

\bibitem[R76]{R76}
{\sc W. Rudin}. 
\newblock{Principles of mathematical analysis.}
\newblock{International Series in Pure and Applied Mathematics (Third ed.), McGraw-Hill, New York, (1976).}





\smallskip 


  \bibitem[SS04]{SS04}
{\sc V. Sidoravicius} and {\sc A. S. Sznitman}.
\newblock Quenched invariance principles for walks on clusters of percolation or among random conductances,
\newblock{Probability theory and related fields.}, {\bf 129}, 219-244, (2004).

\smallskip




\bibitem[Si82]{Si82}
{\sc Y. Sinai}. 
\newblock{The limiting behavior of a one-dimensional random walk in random environment.}
\newblock{\it Theor. Prob. and Appl.} {\bf 27}, 256-268, (1982). 

\smallskip 


\bibitem[So75]{So75}
{\sc F. Solomon}. 
\newblock{Random walks in random environments.}
\newblock{\it Ann. Probab.} {\bf 3}, 1-31, (1975). 

\smallskip 



\bibitem[S94]{S94}
{\sc A. S. Sznitman}. 
\newblock {Shape theorem Lyapunov exponents and large deviations for Brownian motion in a Poissonian potential.}
\newblock{\it Comm. Pure. Appl. Math.}, {\bf 47}, 1655-1688, (1994). 



\smallskip 





\bibitem[S01]{S01}
{\sc A. S. Sznitman}. 
\newblock{On a class of transient random walks in random environment.}
{\it Ann. Probab.}  {\bf 29}, 724-765 (2001). 

\smallskip 

\bibitem[SZ99]{SZ99}
{\sc A. S. Sznitman} and {\sc M. Zerner}.
\newblock{A law of large numbers for random walks in random environment.}
{\it Ann. Probab.} {\bf 27}:4, 1851-1869 (1999).

\smallskip 







\bibitem[V03]{V03}
{\sc S.R.S. Varadhan}. 
\newblock{Large deviations for random walks in a random environment.}
 {\it Comm. Pure Applied Math.}, {\bf 56}, 1222-1245, (2003). 
 
 \smallskip 


\bibitem[Y08]{Y08}
{\sc A. Yilmaz}. 
\newblock {Quenched large deviations for random walk in random environment.}
\newblock {\it Comm. Pure Appl. Math}, {\bf 62}, Issue 8, 1033- 1075, (2009). 
 
  

\bibitem[Y10]{Y10}
{\sc A. Yilmaz}. 
\newblock{Averaged large deviations for random walk in a random environment.}
\newblock{\it Ann.  Inst. Henri Poincar\'e  Probab. 
  Stat.}, {\bf 46}, 853-868, (2010). 

\smallskip 

\bibitem[Y11]{Y11}
{\sc A. Yilmaz}. 
\newblock{Equality of averaged and quenched large deviations for random walks in random environments in dimensions four and higher.}
{\it Probab. Theory Related Fields}, {\bf 149} (3-4), 463-491, (2011). 


\smallskip 


\bibitem[YZ10]{YZ10}
{\sc A. Yilmaz} and {\sc O. Zeitouni}. 
\newblock{Differing Averaged and Quenched Large Deviations for Random Walks in Random Environments in Dimensions Two and Three.}
{\it Comm. Math. Phys.}, {\bf 300}, 1, 243-271, (2010). 

\smallskip 

\bibitem[Zer98]{Z98}
{\sc M. Zerner}. 
\newblock{Lyapunov exponents and quenched large deviations for multidimensional random walks in random environment.}
\newblock{\it Ann. Prob.}, {\bf 26}, 1446-1476, (1998). 

\smallskip 

\bibitem[Zer02]{Zer02}
{\sc M. Zerner}.
\newblock{A non-ballistic law of large numbers for random walks in i.i.d. random environment.}
\newblock{\it Electron. Comm. Probab.} {\bf 7}, 191-197, (2002)

\end{thebibliography}
\end{document}